\documentclass[12pt,a4paper]{article}

\usepackage{latexsym}
\usepackage{amsmath}
\usepackage{amssymb}
\usepackage{amsthm}
\usepackage{amscd}
\usepackage{mathrsfs}
\usepackage[all]{xy}
\usepackage{graphicx}
\usepackage{comment}
\usepackage{arydshln}

\newtheorem{thm}{Theorem}[section]
\newtheorem{prop}[thm]{Proposition}
\newtheorem{defn}[thm]{Definition}
\newtheorem{lem}[thm]{Lemma}
\newtheorem{cor}[thm]{Corollary}
\newtheorem{conj}[thm]{Conjecture}
\newtheorem{rem}[thm]{Remark}
\newtheorem{eg}[thm]{Example}

\newtheorem{thmi}{Theorem}

\theoremstyle{remark}

\makeatletter

\@addtoreset{equation}{section}
\makeatother

\makeatletter

\@addtoreset{equation}{section}
\makeatother

\makeatletter
\newcommand{\subsubsubsection}{\@startsection{paragraph}{4}{\z@}%
 {1.0\Cvs \@plus.5\Cdp \@minus.2\Cdp}%
 {.1\Cvs \@plus.3\Cdp}%
 {\reset@font\sffamily\normalsize}
 }
\makeatother
\DeclareMathOperator{\Gal}{Gal}
\DeclareMathOperator{\id}{id}

\DeclareMathOperator{\Hom}{Hom}
\DeclareMathOperator{\Ext}{Ext}
\DeclareMathOperator{\Rep}{Rep}

\DeclareMathOperator{\Aut}{Aut} 
 
\DeclareMathOperator{\Ind}{Ind}

\DeclareMathOperator{\Spa}{Spa}

\DeclareMathOperator{\Sym}{Sym}
\DeclareMathOperator{\IC}{IC}
\DeclareMathOperator{\Sht}{Sht}
\DeclareMathOperator{\Spd}{Spd}
\DeclareMathOperator{\Perf}{Perf}

\DeclareMathOperator{\Gr}{Gr}
\DeclareMathOperator{\GL}{GL}
\DeclareMathOperator{\Bun}{Bun}
\DeclareMathOperator{\lis}{lis}
\DeclareMathOperator{\cone}{cone}
\DeclareMathOperator{\fib}{fib}
\DeclareMathOperator{\ad}{ad}

\DeclareMathOperator{\sHom}{\mathscr{H}\!\mathit{om}}
\DeclareMathOperator{\cHck}{\mathcal{H}\hspace{-0.1em}\mathit{ck}}

\newcommand{\bA}{\mathbb{A}}
\newcommand{\bB}{\mathbb{B}}
\newcommand{\bC}{\mathbb{C}}
\newcommand{\bD}{\mathbb{D}}

\newcommand{\bF}{\mathbb{F}}

\newcommand{\bL}{\mathbb{L}}

\newcommand{\bQ}{\mathbb{Q}}

\newcommand{\bZ}{\mathbb{Z}}

\newcommand{\cC}{\mathcal{C}}

\newcommand{\cF}{\mathcal{F}}
\newcommand{\cG}{\mathcal{G}}
\newcommand{\cH}{\mathcal{H}}

\newcommand{\cM}{\mathcal{M}}

\newcommand{\cO}{\mathcal{O}}
\newcommand{\cP}{\mathcal{P}}

\newcommand{\cS}{\mathcal{S}}

\newcommand{\cY}{\mathcal{Y}}

\newcommand{\sE}{\mathscr{E}}

\newcommand{\ol}{\overline}
\newcommand{\ul}{\underline}
\newcommand{\wh}{\widehat}
\newcommand{\wt}{\widetilde}
\newcommand{\lra}{\longrightarrow}
\newcommand{\xra}{\xrightarrow}

\newcommand{\cf}{\textit{cf.\ }}

\begin{document}

\title
%[]
{Convolution morphisms and Kottwitz conjecture} 

\author{Naoki Imai}
\date{}
\maketitle
\begin{abstract}
We define etale cohomology of the 
moduli spaces of mixed characteristic local shtukas 
so that it gives smooth representations 
including the 
case where the relevant elements of the Kottwitz set are both non-basic. 
Then we relate the etale cohomology 
of different moduli spaces 
of mixed characteristic local shtukas 
using convolution morphisms, duality morphisms and 
twist morphisms. 
As an application, we show the Kottwitz conjecture 
in some new cases including the cases for 
all inner forms of $\mathrm{GL}_3$ and minuscule cocharacters. 
We study also some non-minuscule cases and 
show that 
the Kottwitz conjecture is true for 
any inner form of $\mathrm{GL}_2$ and 
any cocharacter if the Langlands parameter is 
cuspidal. 
On the other hand, 
we show that 
the Kottwitz conjecture does not hold as it is 
in non-minuscule cases 
if the Langlands parameter is not cuspidal. 
Further, we show that a generalization of 
the Harris--Viehmann conjecture for the moduli spaces 
of mixed characteristic local shtukas 
does not hold in Hodge--Newton irreducible cases. 
\end{abstract}


\section*{Introduction}

The Kottwitz conjecture says that 
etale cohomology of Rapoport--Zink spaces 
or more generally local Shimura varieties 
realize the local Langlands correspondence 
(\cf \cite[Conjecture 5.1]{RapNAper}, 
\cite[Conjecture 7.4]{RVlocSh}). 
In \cite{ScWeBLp}, Scholze 
constructs local Shimura varieties 
as special cases of moduli spaces of 
mixed characteristic local shtukas. 
The Kottwitz conjecture makes sense also for 
the moduli spaces of 
mixed characteristic local shtukas. 
A weak version of the conjecture is studied by 
Hansen--Kaletha--Weinstein in \cite{HKWKotloc}. 
In the weak version, 
we ignore the action of the Weil groups 
and have an equality up to 
representations 
which have trace $0$ on regular elliptic elements. 

Let $p$ be a prime number. 
Let $G$ be a connected reductive group over a $p$-adic number field $F$. 
For $b , b' \in G(\breve{F})$ and 
a system $\mu_{\bullet}=(\mu_1, \ldots , \mu_m)$ of cocharacters of $G$, 
we define a moduli space $\Sht_{b,b'}^{\mu_{\bullet}}$ 
of mixed characteristic local shtukas. See \S \ref{sec:Moduli} 
for the precise definition. 

In this paper, we introduce convolution morphisms, duality morphisms and 
twist morphisms between 
moduli spaces of mixed characteristic local shtukas. 
The convolution morphism is related to 
a convolution morphism on affine Grassmannians. 
Using these morphisms and 
the convolution products in the geometric Satake equivalence for 
$B_{\mathrm{dR}}^+$-Grassmannians, 
we relate the etale cohomology 
of different moduli spaces 
of mixed characteristic local shtukas. 
More concretely, we show the following: 
\begin{thmi}[Corollary \ref{cor:formula}]\label{thmi:formula}
Assume that $G$ is quasi-split and take a Borel pair 
$T \subset B$ of $G$. 
Let $\mu_{\bullet}=(\mu_1 , \ldots , \mu_m)$ be 
a system of dominant cocharacters of $T$ and 
$b_0 , b_m \in G(\breve{F})$. 
Let $E$ be a finite extension of $F$ 
containing the fields of definition of $\mu_i$ for $1 \leq i \leq m$. 
We have 
\begin{align*}
 \sum_{([b_i]
 )_{1 \leq i \leq m-1} \in I_{b_0,b_m}^{\mu_{\bullet}}} 
 H_* 
 \left( \prod_{i=1}^{m-1} G_{b_i}(F), 
 \bigotimes_{1 \leq i \leq m}  H_{\mathrm{c}}^* 
 (\Sht_{b_{i-1},b_i}^{\mu_i}) \otimes 
 \bigotimes_{1 \leq i \leq m-1} \delta_{b_i} 
 \right) & \\ 
 =
 \sum_{\lambda \in X_*(T)^+/\Gamma} V_{\mu_{\bullet}}^{\lambda} & \otimes 
 H_{\mathrm{c}}^* 
 (\Sht_{b_0,b_m}^{\lambda}) 
\end{align*}
as virtual representations of 
$G_{b_0}(F) \times G_{b_m}(F) \times W_E$, 
where $I_{b_0,b_m}^{\mu_{\bullet}}$ is a finite set defined in 
\S \ref{sec:coh}. 
\end{thmi}
We note that even if $b_0$ and $b_m$ are basic, 
non-basic elements appear in 
$I_{b_0,b_m}^{\mu_{\bullet}}$ and 
there are contributions from cohomology of non-basic 
moduli spaces of local shtukas. 
For a derived category version 
of the above statement, see Proposition \ref{prop:decomp}. 

As an application of Theorem \ref{thmi:formula} (or its derived category version) 
together with duality morphisms and twist morphisms, 
we show new cases of 
the Kottwitz conjecture for 
the moduli spaces of 
mixed characteristic local shtukas. 
In particular, we show 
the following: 

\begin{thmi}[Corollary \ref{cor:minu}]\label{thmi:GL3}
Let $G$ be an inner form of $\GL_3$ over $F$. 
Let $(G,b,\mu)$ be a local shtuka datum such that 
$\mu$ is minuscule and $b$ is basic. 
Let $\varphi \colon W_F \to {}^L\GL_3$ 
be a discrete local L-parameter. 
Let $\pi$ and $\pi_b$ be the 
irreducible smooth representations of 
$G(F)$ and $G_b(F)$ corresponding to 
$\varphi$ via the local Langlands correspondence. 
Then we have 
\[
 \cH^* \left( R\Hom_{G(F)} 
 \left( R\Gamma_{\mathrm{c}} (\Sht_{1,b}^{\mu}), \pi 
 \right) \right) \simeq 
 \pi_{b} \boxtimes (r_{\mu} \circ \varphi) 
\]
as representations of $G_{b}(F) \times W_F$. 
\end{thmi}

It is remarkable that the proof of Theorem \ref{thmi:GL3} 
requires moduli spaces of local shtukas for non-minuscule cocharacters, 
even though the statement involves only minuscule cocharacters: 
Using a derived category version of Theorem \ref{thmi:formula}, 
we can calculate a sum of cohomology of 
moduli spaces of local shtukas for a minuscule cocharacter and 
a non-minuscule cocharacter. 
Then we separate them into each term using the duality isomorphism. 
We also note that it is essential to introduce convolution morphisms for 
moduli spaces of mixed characteristic local shtukas with multiple legs 
in \S \ref{sec:conv}, 
because we use it in the proof of a compatibility result, 
Proposition \ref{prop:invcomp}, 
which plays an important role in the proof of Theorem \ref{thmi:GL3}. 

Theorem \ref{thmi:formula} is useful also for studying 
non-minuscule cases. 
We give inductive formulas that enable us to calculate 
the cohomology of moduli spaces of 
local shtukas for 
inner forms of $\GL_2$. 
We can summarize the results in 
\S \ref{sec:ind} as the following theorem: 

\begin{thmi}\label{thmi:Ind}
Let $G$ be an inner form of $\GL_2$ over $F$. 
Let $(G,b,\mu)$ be a local shtuka datum. 
Let $\rho$ be a discrete series representation of 
$G_b(F)$. 
We put 
\[
 H_{\mathrm{c}}^{\bullet} (\Sht_{1,b}^{\mu})[\rho] = 
 \sum_{i,j \in \bZ} (-1)^{i+j} \Ext_{G_b(F)}^i 
 \left( R^j\Gamma_{\mathrm{c}} (\Sht_{1,b}^{\mu}), \rho 
 \right) . 
\]
Then we can calculate 
$H_{\mathrm{c}}^{\bullet} (\Sht_{1,b}^{\mu})[\rho]$ 
by inductive formulas. 
In particular, when $b$ is basic, we see the following: 
\begin{enumerate}
\item 
The Kottwitz conjecture for $\Sht_{1,b}^{\mu}$ 
holds if the L-parameter is cuspidal or 
$G$ is not quasi-split. 
\item 
The Kottwitz conjecture for $\Sht_{1,b}^{\mu}$ 
does not hold in general 
if the L-parameter is not cuspidal and 
$G$ is quasi-split. 
\end{enumerate}
\end{thmi}

We note that in the first statement of Theorem \ref{thmi:Ind}, 
$b$ can be non-basic and $\mu$ can be non-minusucle. 
Even if we are interested only in 
$H_{\mathrm{c}}^{\bullet} (\Sht_{1,b}^{\mu})[\rho]$ for a basic $b$, 
the inductive formulas for the calculation of 
$H_{\mathrm{c}}^{\bullet} (\Sht_{1,b}^{\mu})[\rho]$ 
involve 
moduli spaces of local shtukas for non-basic elements. 
Therefore it is important to study 
non-basic cases at the same time. 

We note that Theorem \ref{thmi:Ind} is compatible with the result in 
\cite{HKWKotloc}, 
since the error term involves only 
representations 
which have trace $0$ on regular elliptic elements. 
We remark also that this error term 
supports that 
the expectation \cite[Remark 4.6]{FarGover} 
in the geometrization of the local Langalnds correspondence 
is true. 

Further, we see that the 
Harris--Viemann conjecture for 
the moduli spaces of 
mixed characteristic local shtukas 
does not hold as it is 
in Example \ref{eg:(3,0)} and 
Remark \ref{rem:HV30}. 
We note that 
Harris--Viemann conjecture for 
the moduli spaces of 
mixed characteristic local shtukas is proved in 
\cite{GINsemi} and \cite{HanHarr} 
under the Hodge--Newton reducibility condition. 
On the other hand, 
the Hodge--Newton reducibility condition is not satisfied in 
Example \ref{eg:(3,0)}. 

In \S \ref{sec:Shell}, we collect results on 
relative homologey and the geometric Satake correspondence. 
In \S \ref{sec:Moduli}, 
we give a definition of a 
moduli space of mixed characteristic local shtukas. 
The definition which we give here is slightly different 
from that in \cite{ScWeBLp}. 
Our definition is suitable 
to construct convolution morphisms between 
moduli spaces of mixed characteristic local shtukas 
in \S \ref{sec:conv}. 
In \S \ref{sec:Twi}, we construct a 
twist morphism between 
moduli spaces of mixed characteristic local shtukas, 
which has an origin in 
the twist of a vector bundle by a line bundle.  
In \S \ref{sec:coh}, 
we discuss a relation between cohomology of different 
moduli spaces of mixed characteristic local shtukas 
using convolution morphisms. 
In \S \ref{sec:Dual}, 
we construct a duality morphism, 
which has an origin in 
the dual of a vector bundle. 
In \S \ref{sec:Kot}, 
we give an application to the Kottwitz conjecture. 
In \S \ref{sec:ind}, 
we give some inductive formulas on cohomology and 
discuss more about the Kottwitz conjecture 
in non-minuscule cases. 

After we put a former version of this paper on arXiv, 
a preprint \cite{HansclocSh} by Hansen 
appeared, where 
a cohomology version of 
Theorem \ref{thmi:GL3} is proved for 
cuspidal local L-parameters of $\GL_n$ 
using a result in \cite{AnLBAvGLn}. 
A merit of Theorem \ref{thmi:GL3} is that 
it works for discrete local L-parameters. 

\subsection*{Acknowledgements}
The author is grateful to Peter Scholze for 
answering questions. 
He would like to thank David Hansen for helpful comments. 
He thank also 
Teruhisa Koshikawa for helpful comments and various discussions. 
This work was supported by JSPS KAKENHI Grant Number 18H01109. 

\subsection*{Notation}
For a field $F$, let $\Gamma_F$ 
denote the absolute Galois group of $F$. 
For a non-archimedean local field $F$, let 
$\breve{F}$ denote the 
completion of the maximal unramified 
extension of $F$. 
For an object $X_Y$ over an object $Y$, 
its base change by the morphism $Y' \to Y$ 
is denoted by $X_{Y'}$. 

\section{Sheaves in $\ell$-adic coefficients}\label{sec:Shell}
\subsection{Relative homology}\label{ssec:Relh}
Let $p$ be a prime number. 
Let $\Lambda$ be a solid $\wh{\bZ}^p$-algebra. 
For a small v-stack $X$, 
we define $D_{\blacksquare}(X,\Lambda)$ as 
\cite[Definition VII.1.17]{FaScGeomLLC}. 
There is a symmetric monoidal structure 
$- \stackrel{\blacksquare}{\otimes}\!\!{}^{\bL}_{\Lambda} -$ on 
$D_{\blacksquare}(X,\Lambda)$ constructed by 
\cite[Proposition VII.2.2]{FaScGeomLLC}. 
In the sequel, we simply write 
$\otimes^{\bL}_{\Lambda}$ for 
$\stackrel{\blacksquare}{\otimes}\!\!{}^{\bL}_{\Lambda}$. 
For a morphism $f \colon X \to Y$ of small v-stacks, 
let 
\[
 f_{\natural} \colon 
 D_{\blacksquare}(X,\Lambda) \to D_{\blacksquare}(Y,\Lambda) 
\] 
be a left adjoint to 
$f^* \colon 
 D_{\blacksquare}(Y,\Lambda) \to D_{\blacksquare}(X,\Lambda)$ 
constructed by 
\cite[Proposition VII.3.1]{FaScGeomLLC}. 

The following lemma is already known (\cf the proof of \cite[Proposition VII.6.3]{FaScGeomLLC}). 

\begin{lem}\label{lem:limtor}
Let $f \colon X \to Y$ be a quasi-compact, quasi-separated morphism 
of small v-stacks. 
Assume that $\Lambda =\varprojlim_{n \in I} \bZ /n\bZ$, 
where $I$ is a filtered set of positive integers which are prime to $p$. 
Then we have 
\[
 f_{\natural} \Lambda \simeq 
 \varprojlim_{n \in I} f_{\natural} (\bZ/n\bZ) . 
\]
\end{lem}
\begin{proof}
We recall a proof. 
We may assume that $Y$ is a spatial diamond. 
Then $\Lambda$ is a pseudo-coherent object on $X$ by the assumption on $f$. 
Since $f_{\natural}$ preserves pseudo-coherent objects, 
$f_{\natural} \Lambda$ is also a pseudo-coherent object. 
Since each cohomology sheaf of 
$f_{\natural} \Lambda$ is a finitely presented solid sheaf, 
we have 
\[
 f_{\natural} \Lambda \simeq 
 \varprojlim_{n \in I} 
 (f_{\natural} \Lambda \otimes^{\bL}_{\Lambda} \bZ/n\bZ ) 
 \simeq \varprojlim_{n \in I} f_{\natural} (\bZ/n\bZ) 
\]
by \cite[Theorem VII.1.3, Proposition VII.3.1]{FaScGeomLLC}. 
\end{proof}

\begin{lem}\label{lem:uniop}
Let $f \colon X \to Y$ be a morphism 
of small v-stacks. 
Let $\cF$ be a solid $\wh{\bZ}^p$-sheaf on $X$. 
Let $\{ U_i \}_{i \in I}$ 
be a filtered direct system of quasi-compact open substacks of $X$ such that $X=\bigcup_{i \in I} U_i$. 
Let $f_i$ and $\cF_i$ be the restriction to $U_i$ of $f$ and $\cF$ 
for $i \in I$. 
Then we have 
\[
 f_{\natural} \cF \simeq \varinjlim_{i \in I} {f_i}_{\natural} \cF_i . 
\]
\end{lem}
\begin{proof}
Let $j_i \colon U_i \to X$ be the inclusion for $i \in I$. 
Since $f_{\natural}$ commutes with 
a direct limit, it suffices to show 
$\cF \simeq \varinjlim_{i \in I} {j_i}_{\natural} \cF_i$. 
By the projection formula, we may assume that 
$\cF =\wh{\bZ}^p$. 
For any solid $\wh{\bZ}^p$-sheaf $\cG$ on $X$, 
we have 
\begin{align*}
 \Hom (\varinjlim_{i \in I} {j_i}_{\natural} \wh{\bZ}^p,\cG) &\simeq 
 \varprojlim_{i \in I} \Hom ( {j_i}_{\natural} \wh{\bZ}^p, \cG ) \simeq 
 \varprojlim_{i \in I} \cG (U_i) \simeq \cG (X) \simeq 
 \Hom (\wh{\bZ}^p,\cG). 
\end{align*}
Hence we obtain the claim. 
\end{proof}

\begin{lem}\label{lem:cohAD}
Let $F$ be a non-archimedean field with residue characteristic $p$. 
Let $d$ be a positive integer. 
\begin{enumerate}
\item\label{en:ballcont} 
Let 
\[
 f \colon \bigl( \Spa (\cO_F [[x_1^{1/p^{\infty}},\ldots ,x_d^{1/p^{\infty}}]]) \times_{\Spa (\cO_F)} \Spa (F) \bigr)^{\diamond} \to \Spa (F)^{\diamond} 
\]
be the natural morphism. 
Then we have 
$f_{\natural} \Lambda \simeq \Lambda$. 
Further, the geometric Frobenius morphism 
$x_i \mapsto x_i^p$ induces the multiplication by 
$p^d$ on $f_{\natural} \Lambda$. 
\item\label{en:Adcont} 
Let 
\[
 f \colon (\bA_F^d)^{\diamond} \to \Spa (F)^{\diamond} 
\]
be the natural morphism. 
Then we have 
$f_{\natural} \Lambda \simeq \Lambda$. 
Further, the geometric Frobenius morphism 
$x_i \mapsto x_i^p$ induces the multiplication by 
$p^d$ on $f_{\natural} \Lambda$. 
\end{enumerate}
\end{lem}
\begin{proof}
We show the first claim of \ref{en:ballcont}. 
We may assume that 
$\Lambda=\wh{\bZ}^p$ and 
$F$ is algebraically closed of characteristic $p$. 
We write $\Spa (\cO_F [[x_1^{1/p^{\infty}},\ldots ,x_d^{1/p^{\infty}}]]) \times_{\Spa (\cO_F)} \Spa (F)$ 
as a union of affinoids isomorphic to 
$\Spa (F \langle x_1^{1/p^{\infty}},\ldots ,x_d^{1/p^{\infty}} \rangle)$. 
By Lemma \ref{lem:uniop}, 
it is reduced to show that 
$g_{\natural} \wh{\bZ}^p \simeq \wh{\bZ}^p$ for 
\[
 g \colon \Spa (F \langle x_1^{1/p^{\infty}},\ldots ,x_d^{1/p^{\infty}} \rangle  )
 \to \Spa (F). 
\]
By Lemma \ref{lem:limtor} and 
\cite[Proposition VII.5.2]{FaScGeomLLC}, 
the claim follows from that 
$g_!(\bZ/n\bZ)\simeq (\bZ/n\bZ) (-d)[-2d]$ for any integer $n$ prime to $p$. 
The claim on the geometric Frobenius morphism follows 
from the case for $g_!(\bZ/n\bZ)$. 

We can show the claim \ref{en:Adcont} similarly. 
\end{proof}

Let $\ell$ be a prime number different from $p$. 

\begin{lem}\label{lem:torhom}
Let $G$ be a locally pro-$p$ group. 
Let $\cH(G)$ be the Hecke algebra of $G$ 
with coefficients in $\Lambda$. 
Let $f \colon X \to Y$ be a morphism of small v-stacks 
which is a $G$-torsor. 
For a pro-$p$ open subgroup $K$ of $G$, 
let $f_K \colon X/K \to Y$ be the morphism induced by $f$.  
%and $\cH(G,K)$ the Hecke algebra of $G$ with respect to $K$. 
Let $g \colon Y \to Z$ be a morphism of small v-stacks. 
The morphisms $f_{K}^*$ and $(g \circ f_K)_{\natural}$ induce 
\begin{equation*}
\varinjlim_{K} (g \circ f_K)_{\natural} f_{K}^* \colon D_{\blacksquare}(Y,\Lambda) \to D_{\blacksquare}(Z,\cH(G))
\end{equation*}
%\begin{align*}
%	&f_{K}^* \colon D_{\blacksquare}(Y,\Lambda) \to D_{\blacksquare}(X/K,\cH(G,K)), \\ 
%	&(g \circ f_K)_{\natural} \colon D_{\blacksquare}(X/K,\cH(G,K)) 	\to D_{\blacksquare}(Z,\cH(G,K)\times \cH(G,K)) \to 	D_{\blacksquare}(Z,\cH(G,K))
%\end{align*}
%where the last morphism in the second line is given by 
%the diagonal embedding $\cH(G,K) \to \cH(G,K)\times \cH(G,K)$. 
%Then we have 
%\[ \bigl( \varinjlim_{K}(g \circ f_K)_{\natural} (f_K^* A) \bigr) \otimes_{\cH(G)}^{\bL} \Lambda \simeq  g_{\natural} A \]
%for . 
\begin{enumerate}
\item\label{en:ffAA}
For $A \in D_{\blacksquare}(Y,\Lambda)$, we have 
\[
  ( \varinjlim_{K} f_{K,\natural} f_K^* A ) \otimes_{\cH(G)}^{\bL} \Lambda \cong A. 
\]
\item\label{en:fVH} 
Assume that $A \in D_{\blacksquare}(Y,\Lambda)$ is obtained from $V \in D^{\mathrm{b}}(G,\Lambda)$. 
Then we have 
\[
\bigl( \varinjlim_{K}(g \circ f_K)_{\natural} (\Lambda) \otimes_{\Lambda} V \bigr) \otimes_{\cH(G)}^{\bL} \Lambda \simeq 
g_{\natural} A . 
\]
\end{enumerate}
\end{lem}
\begin{proof}
%For a compact open subgroup $K$ of $G$, we have 
%\begin{align*}
% f_{K,\natural} \colon D_{\blacksquare}(X/K,\ol{\bQ}_{\ell}) \to D_{\blacksquare}(Y,\cH(G,K)), \\ 
% g_{\natural} \colon D_{\blacksquare}(Y,\cH(G,K))  \to D_{\blacksquare}(Z,\cH(G,K))
%\end{align*}
%where $\cH(G,K)$ is the Hecke algebra of $G$ with respect to $K$. 
\ref{en:ffAA} 
We have 
\begin{align*}
	 ( \varinjlim_{K} f_{K,\natural} f_K^* A ) \otimes_{\cH(G)}^{\bL} \Lambda  
	\cong  
	 ((\varinjlim_{K} f_{K,\natural} \Lambda) \otimes_{\Lambda}^{\bL} A ) \otimes_{\cH(G)}^{\bL} \Lambda 
	\cong 
	 ((\varinjlim_{K} f_{K,\natural} \Lambda) \otimes_{\cH(G)}^{\bL} \Lambda) \otimes_{\Lambda}^{\bL} A . 
\end{align*}
Hence it suffices to show that the natural morphism 
$(\varinjlim_{K} f_{K,\natural} \Lambda) \otimes_{\cH(G)}^{\bL} \Lambda \to \Lambda$ 
is an isomorphism. We can check this v-locally on $Y$ 
by \cite[Proposition VII.3.1 (iii)]{FaScGeomLLC}. 
Hence the claim follows. \\
\ref{en:fVH} 
The morphism $g_{\natural}$ induces  
\[
 g_{\natural} \colon D_{\blacksquare}(Y,\cH(G)) 
 \to D_{\blacksquare}(Z,\cH(G)). 
\]
By \cite[Proposition VII.3.1 (i)]{FaScGeomLLC}, 
we have 
\begin{align*}
\bigl( \varinjlim_{K}(g \circ f_K)_{\natural} (\Lambda) & \otimes_{\Lambda} V \bigr) \otimes_{\cH(G)}^{\bL} \Lambda \cong 
g_{\natural} \bigl( \varinjlim_{K}f_{K,{\natural}} (\Lambda) \otimes_{\Lambda} V \bigr) \otimes_{\cH(G)}^{\bL} \Lambda \\ 
&\cong 
g_{\natural} \Bigl( \bigl( \varinjlim_{K}f_{K,{\natural}} (\Lambda) \otimes_{\Lambda} V \bigr) \otimes_{\cH(G)}^{\bL} \Lambda \Bigr) 
\cong g_{\natural} \Bigl( \bigl( \varinjlim_{K}f_{K,{\natural}} (V)  \bigr) \otimes_{\cH(G)}^{\bL} \Lambda \Bigr) . 
\end{align*}
Combined with \ref{en:ffAA}, it remains to show 
\[
 \varinjlim_{K}f_{K,{\natural}} (V) \cong 
 \varinjlim_{K} f_{K,\natural} f_K^* A . 
\]
We can check that the morphism 
\[
 \varinjlim_{K}f_{K,{\natural}} (V) \to 
\varinjlim_{K} f_{K,\natural} f_K^* A 
\]
induced by $V \twoheadrightarrow V^K \hookrightarrow f_K^* A$ is an isomorphism. 
\end{proof}

Let $\Lambda$ be a $\bZ_{\ell}$-algebra. 
For an Artin v-stack $X$, let 
$D_{\lis}(X,\Lambda))$ be the category defined in 
\cite[Definition VII.6.1]{FaScGeomLLC}. 

\begin{lem}\label{lem:cslis}
Let $f \colon X \to Y$ be an $\ell$-cohomologically smooth morphism. 
\begin{enumerate}
\item\label{en:lispres} 
We have $f_{\natural} ( D_{\lis}(X,\Lambda)) \subset D_{\lis}(Y,\Lambda)$. 
\item\label{en:pullis}
For $A \in D_{\blacksquare}(Y,\Lambda)$, we have 
$(f^*A)^{\lis} \cong f^*(A^{\lis})$. 
\end{enumerate}
\end{lem}
\begin{proof}
The claim \ref{en:lispres} follows from \cite[Definition VII.6.1]{FaScGeomLLC}. For $B \in D_{\blacksquare}(Y,\Lambda)$, we have 
\begin{align*}
 \Hom (B,(f^*A)^{\lis}) &\cong \Hom (B,f^*A) \cong \Hom (f_{\natural}(B),A) \\ 
 &\cong  
 \Hom (f_{\natural}(B),A^{\lis}) \cong \Hom (B,f^*(A^{\lis})), 
\end{align*}
where we use \ref{en:lispres} at the third isomorphism. 
Hence the claim \ref{en:pullis} follows. 
\end{proof}

\begin{lem}\label{lem:lisbc}
	Let  
	\[
	\xymatrix{
		X' 
		\ar[r]^-{f'} \ar[d]^-{g'} & 
		Y' \ar[d]^-{g} \\ 
		X \ar[r]^-{f} & Y 
	}
	\]
	be a cartesian diagram of small v-stack. 
	Assume that $g$ is $\ell$-cohomologically smooth. 
	Then we have 
	\[
	g^* Rf_{\lis*} A \cong  Rf'_{\lis*} g'^* A
	\]
	for $A \in D_{\lis}(X,\Lambda)$. 
\end{lem}
\begin{proof}
	This follows from \cite[Proposition VII.2.4]{FaScGeomLLC} 
	and Lemma \ref{lem:cslis}. 
\end{proof}

\begin{lem}\label{lem:pullhomlis}
Let $f \colon X \to Y$ be an $\ell$-cohomologically smooth morphism. 
Let $A,B \in D_{\lis}(Y,\Lambda)$. 
Then we have 
$f^* R\sHom_{\lis} (A,B) \cong R\sHom_{\lis} (f^* A ,f^* B)$. 
\end{lem}
\begin{proof}
This follows from \cite[Proposition VII.2.4]{FaScGeomLLC}
and Lemma \ref{lem:cslis}. 
\end{proof}

\begin{lem}\label{lem:pushhomlis}
Let $f \colon X \to Y$ be a morphism. 
Let $A \in D_{\lis}(X,\Lambda)$ and $B \in D_{\lis}(Y,\Lambda)$. 
\begin{enumerate}
	\item\label{en:RHflis}
	We have $R\sHom_{\lis} (B,Rf_{\lis *}(A))\cong Rf_{\lis *} R\sHom_{\lis} (f^* B,A)$. 
		\item\label{en:RHfnat} 
If $f$ is $\ell$-cohomologically smooth, then 
we have 
\[
R\sHom_{\lis} (f_{\natural}(A),B)\cong Rf_{\lis *} R\sHom_{\lis} (A,f^* B). 
\]
	\end{enumerate}
\end{lem}
\begin{proof}
\ref{en:RHflis} For $C\in D_{\lis}(Y,\Lambda)$, we can check 
\begin{align*}
R\Hom (C,R\sHom_{\lis} (B,Rf_{\lis *}(A))) \cong 
R\Hom (C,Rf_{\lis *} R\sHom_{\lis} (f^* B,A))
\end{align*}
by adjoint. The claim \ref{en:RHfnat} is proved similarly. 
\end{proof}

For an $\ell$-cohomologically smooth morphism $f \colon X \to Y$, 
we put 
\[
D_f=(\varprojlim_{n} Rf^! (\bZ/\ell^n \bZ)) \otimes_{\bZ_{\ell}} \Lambda 
\] 
and 
%define $f_! \colon D_{\lis}(X,\Lambda) \to D_{\lis}(Y,\Lambda)$ by 
\[
 f_! (A)=f_{\natural} (A \otimes D_f^{-1}) 
\] 
for $A \in D_{\lis}(X,\Lambda)$. 
For an $\ell$-cohomologically smooth morphism $f \colon X \to *$, 
we write $D_X$ for $D_f$. 
For $f \colon X \to *$ and $A \in D_{\lis}(X,\Lambda)$, 
we put 
$R\Gamma_{\natural}(X,A)=f_{\natural}(A)$. 
For $f \colon X \to \Spa C$ and $A \in D_{\lis}(X,\Lambda)$ where $C$ is an algebraically closed non-archimedean field of characteristic $p$, 
we put 
$R\Gamma_{\natural,C}(X,A)=f_{\natural}(A)$. 

\subsection{Geometric Satake equivalence}\label{ssec:GSat}

We recall the geometric Satake equivalence for 
$B_{\mathrm{dR}}^+$-Grassmannians by Fargues--Scholze 
(\cf \cite[VI, IX]{FaScGeomLLC}). 

Let $\bC_p$ be the completion of the algebraic closure of 
$\bQ_p$. 
Let $F$ be a finite extension of $\bQ_p$ in $\bC_p$ 
with the residue field $\bF_q$. 
For an algebraic field extension $k$ of $\bF_q$, 
let $\Perf_k$ denote the 
category of perfectoid spaces over $k$ 
with $v$-topology 
in the sense of \cite[\S 8]{SchEtdia}. 

Let $G$ be a connected reductive group over $F$. 
We define $v$-sheaves $LG$ and $L^+G$ 
over $\Spd \bQ_p$ by sending 
$S =\Spa (R,R^+) \in \Perf_{\bF_q}$ with an untilt 
$S^{\sharp}=\Spa (R^{\sharp},R^{\sharp,+})$ 
to 
$B_{\mathrm{dR}}(R^{\sharp})$ and 
$B_{\mathrm{dR}}^+(R^{\sharp})$, 
where 
$B_{\mathrm{dR}}(R^{\sharp})$ and 
$B_{\mathrm{dR}}^+(R^{\sharp})$ are defined as in 
\cite[Definition 1.32]{FarGover}. 
We put $\Gr_{G}=LG/L^+G$ and 
\[
 \cHck_G =[L^+G \backslash LG/L^+G] . 
\]
%Let $P_{L^+G}(\Gr_G)$ be the category of 
%$L^+ G$-equivariant $\ol{\bQ}_{\ell}$-perverse sheaf on $\Gr_G$. 
For $A_1, A_2 \in D_{\blacksquare}(\cHck_G,\Lambda)$, 
let 
$A_1 \star A_2$ denote the convolution product of 
$A_1$ and $A_2$. 
Let $Q$ be a finite quotient of $W_F$ such that 
the action of $W_F$ on $\wh{G}$ factors through $Q$. 
Let 
\[
 \cS' \colon \Rep_{\Lambda} (\wh{G} \rtimes Q) 
 \lra D_{\blacksquare}(\cHck_G,\Lambda) 
\]
denote the functor that gives the 
geometric Satake equivalence (\cf \cite[IX.2]{FaScGeomLLC}). 
This functor is symmetric monoidal functor by the construction 
(\cf \cite[Proposition VI.10.2]{FaScGeomLLC}). 

For $V_1, V_2 \in \Rep_{\Lambda} (\wh{G} \rtimes Q)$, 
let 
\[
 c_{V_1, V_2} \colon \cS' (V_1) \star \cS' (V_2) \simeq 
 \cS' (V_2) \star \cS' (V_1) 
\]
be the commutativity constraint uniquely characterized by 
\[
 \xymatrix{
 \cS' (V_1) \star \cS' (V_2) 
 \ar[rr]^-{c_{V_1, V_2}} \ar[d] & & 
 \cS' (V_2) \star \cS' (V_1) \ar[d] \\ 
 \cS' (V_1 \otimes V_2) 
 \ar[rr]^-{S'(\sigma_{V_1,V_2})} & & 
 \cS' (V_2 \otimes V_1) , 
 }
\]
where $\sigma_{V_1,V_2} \colon V_1 \otimes V_2 \to V_2 \otimes V_1$ 
is the isomorphism switching $V_1$ and $V_2$. 

Assume that $\mu \in X_*(T)^+$. 
Let $Q_{\mu} \subset Q$ be the stabilizer of $\mu$. 
Let $r'_{G,\mu}$ be the highest weight $\mu$ irreducible 
representation of $\wh{G} \rtimes Q_{\mu}$. 
We put 
\[
 r_{G,\mu}=\Ind_{\wh{G} \rtimes Q_{\mu}}^{\wh{G} \rtimes Q} r'_{G,\mu}. 
\]
We simply write $r_{\mu}$ for $r_{G,\mu}$ 
if there is no confusion. 
We write $V_{\mu}$ for the representation space of 
$r_{\mu}$. 
We put $\mathrm{IC}'_{\mu}=\cS' (V_{\mu})$. 
We use the same notation 
$\mathrm{IC}'_{\mu}$ for the pullback of 
$\mathrm{IC}'_{\mu}$ to other spaces.

\section{Moduli of local shtukas}\label{sec:Moduli}

Let $S =\Spa (R,R^+) \in \Perf_{\bF_q}$. 
We put 
$W_{\cO_F}(R^+)=W(R^+) \otimes_{W(\bF_q)} \cO_F$. 
Take an topological nilpotent unit $\varpi_R$ in $R$. 
Let $\cY_{(0,\infty)}(S)$ be the adic space 
defined by the condition 
$p \neq 0$ and $[\varpi_R] \neq 0$ 
in $\Spa (W_{\cO_F}(R^+),W_{\cO_F}(R^+))$. 
Then $\cY_{(0,\infty)}(S)$ has an action of 
the $q$-th power Frobenius element 
$\varphi_S$ induced by the $q$-th power map on $R$. 
The quotient 
\[
 X_S = \cY_{(0,\infty)}(S)/ \varphi_S^{\bZ} 
\]
is called the relative Fargues--Fontaine curve for $S$ 
(\cf \cite[Definition 15.2.6]{ScWeBLp}). 
The construction glues together to give 
$X_S$ for any $S \in \Perf_{\bF_q}$. 

We define a continuous map 
\[
 \kappa_{S} \colon \cY_{(0,\infty)}(S) 
 \lra (0,\infty) 
\]
by 
\[
 \kappa_{S} (x) = 
 \frac{\log \lvert [\varpi_R] \rvert_{\wt{x}}}{\log \lvert p \rvert_{\wt{x}}} 
\]
where $\wt{x}$ is the maximal generalization of 
$x \in \cY_{(0,\infty)}(S)$ and 
$\lvert \, \cdot \, \rvert_{\wt{x}}$ denotes the valuation corresponding to 
$\wt{x}$. 
For an interval $I$ in $(0,\infty)$, 
let 
$\cY_{I}(S)$ denote the interior of 
$\kappa_{S}^{-1}(I)$. 
For $S \in \Perf_{\bF_q}$, 
we put 
$\bB (S)=\cO (\cY_{(0,\infty)}(S))$. 
Then $\bB$ is a v-sheaf by \cite[Proposition II.2.1]{FaScGeomLLC}. 

Let $G$ be a connected reductive group over $F$. 
Let $b \in G(\breve{F})$. 
We define an algebraic group $G_b$ over $F$ by 
\[
 G_b (R)=\{ g \in G(\breve{F} \otimes_F R) \mid 
 g (b \sigma \otimes 1) = (b \sigma \otimes 1) g \} 
\]
for any $F$-algebra $R$. 
We define 
a $G$-bundle $\sE_{b,X_S}$ on $X_S$ by 
\[
 (G_{\breve{F}} \times_{\Spa (\breve{F})} \cY_{(0,\infty)}(S)) 
 / ((b\sigma) \times \varphi_S )^{\bZ}. 
\]
If $b'=g^{-1}b\sigma(g)$ for $b,b',g \in G(\breve{F})$, 
then the left multiplication by $g^{-1}$ induces 
an isomorphim 
\begin{equation}\label{eq:tg}
 t_g \colon \sE_{b,X_S} \to \sE_{b',X_S}. 
\end{equation}
We define a sheaf $\wt{J}_b$ on $\Perf_{\ol{\bF}_q}$ 
by 
\[
 \wt{J}_b (S)=\Aut (\sE_{b,X_S}) 
\] 
for $S \in \Perf_{\ol{\bF}_q}$. 
In the sequel, we simply write 
$\sE_{b}$ for $\sE_{b,X_S}$ if there is no confusion. 
We define $\wt{J}_b^{>0}$ as in \cite[III.5]{FaScGeomLLC}. 
Then we have 
$\wt{J}_b=\wt{J}_b^{>0} \rtimes \ul{G_b(F)}$ 
by \cite[Proposition III.5.1]{FaScGeomLLC}. 
If $b$ is basic, we have $\wt{J}_b =\ul{G_b(F)}$. 

Let $b , b' \in G(\breve{F})$. 
Let $\mu_1, \ldots , \mu_m$ be cocharacters of $G$. 
We put $\mu_{\bullet}=(\mu_1, \ldots , \mu_m)$. 
For $1 \leq i \leq m$, 
let $E_i$ be the field of definition of $\mu_i$. 

\begin{defn}
We define the presheaf 
$\Sht_{G,b,b'}^{\mu_{\bullet}}$ 
by sending $S =\Spa (R,R^+) \in \Perf_{\ol{\bF}_q}$ 
to the isomorphism classes of the following objects; 
\begin{itemize}
\item
an untilt $S_i^{\sharp}$ of $S$ over $\breve{E}_i$ 
for $1 \leq i \leq m$, 
\item 
a $G$-torsor $\cP$ on $\cY_{(0,\infty)}(S)$ with an isomorphism 
\[
 \varphi_{\cP} \colon (\varphi_S^* \cP )
 |_{\cY_{(0,\infty)}(S) \setminus \bigcup_{i=1}^m S_i^{\sharp}} 
 \simeq 
 \cP|_{\cY_{(0,\infty)}(S) \setminus \bigcup_{i=1}^m S_i^{\sharp}} 
\]
which is meromorphic along the Cartier divisor 
$\bigcup_{i=1}^m S_i^{\sharp} \subset \cY_{(0,\infty)}(S)$ 
and 
the relative position of 
$\varphi_S^* \cP$ and $\cP$ at $S_i^{\sharp}$ 
is bounded by 
$\sum_{j \mid S_j^{\sharp}=S_i^{\sharp}} \mu_j$ 
at all geometric rank $1$ points for all $1 \leq i \leq m$, 
\item 
an isomorphism 
\[
 \iota_{[r,\infty)} \colon \cP|_{\cY_{[r,\infty)}(S)} \simeq 
 G \times \cY_{[r,\infty)}(S) 
\]
for large enough $r$ under which 
$\varphi_{\cP}$ is identified with $b \times \varphi_S$ 
and 
an isomorphism 
\[
 \iota_{(0,r']} \colon \cP|_{\cY_{(0,r']}(S)} \simeq G \times \cY_{(0,r']}(S) 
\]
for small enough $r'$ under which 
$\varphi_{\cP}$ is identified with $b' \times \varphi_S$ 
\end{itemize}
\end{defn}

If there is no confusion, we simply write 
$\Sht_{b,b'}^{\mu_{\bullet}}$ for $\Sht_{G,b,b'}^{\mu_{\bullet}}$. 
If $\mu_{\bullet}=(\mu)$, 
we simply write 
$\Sht_{G,b,b'}^{\mu}$ for 
$\Sht_{G,b,b'}^{\mu_{\bullet}}$. 
We use similar abbreviations also for other spaces. 

We define the right action of $\wt{J}_b \times \wt{J}_{b'}$ 
on $\Sht_{G,b,b'}^{\mu_{\bullet}}$ by 
\[
 (\iota_{[r,\infty)}, \iota_{(0,r']}) \mapsto 
 (g^{-1} \circ \iota_{[r,\infty)}, g'^{-1} \circ \iota_{(0,r']}) 
\]
for $(g,g') \in \wt{J}_b \times \wt{J}_{b'}$. 

We define 
$\Gr_{G,\Spd E_1 \times \cdots \times \Spd E_m, 
 \leq \mu_{\bullet}}^{\mathrm{tw}}$ 
as in \cite[Definition 23.4.1]{ScWeBLp}. 
It is a spacial diamond by 
\cite[Proposition 23.4.2]{ScWeBLp}. 
We have a morphism 
\[
 \pi_{G,b,b'}^{\mu_{\bullet}} \colon 
 \Sht_{G,b,b'}^{\mu_{\bullet}} \to 
 \Gr_{G,\Spd \breve{E}_1 \times \cdots \times \Spd \breve{E}_m, 
 \leq \mu_{\bullet}}^{\mathrm{tw}}
\]
defined by forgetting $\iota_{(0,r']}$. 
The morphism 
$\pi_{G,b,b'}^{\mu_{\bullet}}$ 
is a $\wt{J}_{b'}$-torsor over a locally spatial subdiamond 
of 
$\Gr_{G,\Spd \breve{E}_1 \times \cdots \times \Spd \breve{E}_m, 
 \leq \mu_{\bullet}}^{\mathrm{tw}}$ 
by \cite[Proposition 11.20]{SchEtdia}. 
Hence, 
$\Sht_{G,b,b'}^{\mu_{\bullet}}$ 
is a diamond by 
\cite[Proposition 11.6]{SchEtdia} and 
\cite[2.5, 2.6.2]{FarGover}. 

We have a natural inversing morphism
\begin{equation}\label{eq:invmor}
 \Sht_{G,b,b'}^{\mu_{\bullet}} \to 
 \Sht_{G,b',b}^{\mu_{\bullet}^{-1}} 
\end{equation}
compatible with the action of $\wt{J}_b \times \wt{J}_{b'}$. 

Let $B(G)$ be the set of $\sigma$-conjugacy classes in 
$G(\breve{F})$. 
We write $B(G)_{\mathrm{bas}}$ 
for the set of the basic elements in 
$B(G)$. 
Let $\mu$ be a cocharacter of $G$. 
We define $B(G,\mu)$ as in \cite[6.2]{KotIsoII}. 

Assume that $G$ is quasi-split. 
We fix subgroups 
$A \subset T \subset B$ of $G$ 
where $A$ is a maximal split torus, 
$T$ is a maximal torus and $B$ is a Borel subgroup. 
We write 
$X_*(A)^+$ and $X_*(T)^+$ 
for the dominant cocharacter of $A$ and $T$. 
For $b \in G(\breve{F})$, 
we define $\nu_b \in X_* (A)_{\bQ}^+$ 
as in \cite[2.2.2]{FarGover} 
using the slope morphism constructed in \cite[4.2]{KotIso}. 
Let $B(G,\mu,[b])$ 
be the set of acceptable neutral elements in $B(G)$ for 
$(\mu,[b])$ (\cf \cite[Definition 4.3]{GINsemi}).

\begin{lem}\label{lem:BGbij}
Assume that $b$ is basic. 
The map 
\[
 G(\breve{F}) \to G(\breve{F}) = G_{b}(\breve{F});\ g \mapsto gb^{-1} 
\]
induces bijections 
$B(G) \to B(G_b)$, 
$B(G)_{\mathrm{bas}} \to B(G_b)_{\mathrm{bas}}$ 
and 
$B(G,\mu,[b]) \to B(G_b,\mu)$. 
\end{lem}
\begin{proof}
The claim follows from the equality 
\[
 (g' g \sigma(g')^{-1}) b^{-1} = 
 g' (g b^{-1}) (b \sigma(g') b^{-1})^{-1} . 
\]
for $g,g' \in G(\breve{F})$. 
\end{proof}

\begin{prop}\label{prop:Mtriv}
Assume that $b'$ is basic. 
We have a natural isomorphism 
\[
 \Sht_{G,b,b'}^{\mu_{\bullet}} \stackrel{\sim}{\lra} 
 \Sht_{G_{b'},bb'^{-1},1}^{\mu_{\bullet}} 
\]
which is compatible with the 
action of $\wt{J}_{b} \times \wt{J}_{b'}$. 
\end{prop}
\begin{proof}
We can view 
$\Sht_{G,b,b'}^{\mu_{\bullet}}$ 
as a moduli space of modifications of $G$-torsors on 
a Fargues--Fontaine curve. 
The category of $G$-torsor is 
equivalent to 
the category of $G_{b'}$-torsor 
on a Fargues--Fontaine curve 
as explained in the proof of 
\cite[Corollary 23.2.3]{ScWeBLp}. 
The claim follows from this equivalence. 
\end{proof}

\begin{rem}
Assume that $b,b'$ are basic and $m=1$. 
Then a weak version of Kottwitz conjecture for 
$\Sht_{G,b,b'}^{\mu_{\bullet}}$ holds 
by \cite[Theorem 1.0.4]{HKWKotloc}, 
Lemma \ref{lem:BGbij} and Proposition \ref{prop:Mtriv}. 
\end{rem}

\begin{rem}
Assume that $b,b'$ are basic and $m=1$. 
Under the isomorphism in Proposition \ref{prop:Mtriv}, 
the inversing morphism \eqref{eq:invmor} 
is identified with 
the Faltings--Fargues isomorphism proved in 
\cite[Corollary 23.2.3]{ScWeBLp}. 
\end{rem}

\begin{lem}\label{lem:Mempty}
Assume that $b'$ is basic. 
If $\Sht_{G,b,b'}^{\mu}$ is not empty, 
then we have $[b] \in B(G,\mu,[b'])$. 
\end{lem}
\begin{proof}
By Proposition \ref{prop:Mtriv}, 
we may assume that $b'=1$ dropping the assumption that 
$G$ is quasi-split. 
Then the claim follows from 
\cite[Proposition 3.5.3]{CaScGenSV}. 
\end{proof}

We define a Weil descent datum of $\wt{J}_b$ by 
\[
 \wt{J}_b \to \wt{J}_{\sigma (b)}=\sigma^* (\wt{J}_b); \ 
 f \mapsto t_b \circ f \circ t_b^{-1} , 
\]
where $t_b$ is defined in \eqref{eq:tg}. 
Let $\rho_G$ denote the half-sum of the 
positive roots of $G$ with respect to $T$ and $B$. 
We put $N_b= \langle 2\rho_G , \nu_b \rangle$. 

\begin{lem}\label{lem:Jbnat}
Let $\Lambda$ be a solid $\wh{\bZ}^p$-algebra. 
Let $f \colon \wt{J}_b^{>0} \to *$ be the structure morphism. 
Then we have an isomorphism 
$f_{\natural}(\Lambda) \simeq \Lambda (N_b)$ compatible with the actions of $W_F$.  
\end{lem}
\begin{proof}
This is proved in the same way as \cite[Lemma 4.17]{GINsemi} 
using Lemma \ref{lem:cohAD} and the definition of the Weil descent datum. 
\end{proof}

Let $\delta_b \colon G_b(F) \to \Lambda^{\times}$ 
be the character obtained by the action of 
$G_b(F)$ on $D_f$, where 
$f \colon \wt{J}_b^{>0} \to *$.

\section{Cohomology of moduli of local shtukas}\label{sec:Twi}
Let $\mu_{\bullet}=(\mu_1 , \ldots , \mu_m) \in (X_* (T)^+)^m$. 
Let $E$ be the field of definition of $\mu_{\bullet}$. 
The space $\Sht_{G,b,b'}^{\mu_{\bullet}}$ 
is the moduli space of 
$(S^{\sharp}, \mathscr{E}_{b} \to \mathscr{E}_{b'})$, 
where $S^{\sharp}$ is an ultilt over $\breve{E}$ and 
$\mathscr{E}_{b} \to \mathscr{E}_{b'}$ is a 
modification bounded by $\mu_{\bullet}$ 
along the Cartier divisor defined by $S^{\sharp}$. 

Let $\bC_p^{\flat}$ denote the 
tilt of $\bC_p$. 
The untilt $\bC_p$ of $\bC_p^{\flat}$ 
determine a morphism 
$\Spa \bC_p^{\flat} \to \Spd \bQ_p$. 
For the arithmetic Frobenius element 
$\sigma_E \in \Gal (E^{\mathrm{ur}}/E)$, 
we take $m$ such that $\sigma_E |_{F^{\mathrm{ur}}}=\sigma^m$ and 
define a Weil descent datum of $\Sht_{G,b,b'}^{\mu_{\bullet}}$ by 
\begin{align*}
 \Sht_{G,b,b'}^{\mu_{\bullet}} 
 &\to 
 \Sht_{G,\sigma^m(b),\sigma^m(b')}^{\mu_{\bullet}} 
 = \sigma_E^* (\Sht_{G,b,b'}^{\mu_{\bullet}});\\ 
 (S^{\sharp}, \mathscr{E}_{b} \stackrel{f}{\to} \mathscr{E}_{b'}) 
 &\mapsto 
 (S^{\sharp}, \mathscr{E}_{\sigma^m(b)} \xra{t_{m,b}^{-1}} \mathscr{E}_{b} \xra{f} \mathscr{E}_{b'} \xra{t_{m,b'}} \mathscr{E}_{\sigma^m(b')})
\end{align*}
where we put 
$t_{m,b}=t_{\sigma^{m-1}(b)} \circ \cdots \circ t_b \colon \mathscr{E}_{b} \to \mathscr{E}_{\sigma^m(b)}$ 
using \eqref{eq:tg}. 

We have fiber products 
\[
 \xymatrix{
	\wt{\cM}_b^{\circ} \ar[r]^-{\wt{j}_{b}}  \ar[d]   & 
 	\wt{\cM}_b \ar[d] & 
 	{* } \ar[l]_-{\wt{i}_b} \ar[d] \\ 
	\cM_b^{\circ} \ar[r]^-{j_{b}}  \ar[d] %\ar[rd]^-{\pi_b^{\circ}}  
	& 
	\cM_b \ar[d]^-{\pi_b} & 
	[*/\ul{G_b(F)}] \ar[l]_-{i_b} \ar[d]^-{h_b} \\ 
	\Bun_G^{<b} \ar[r]^-{j^b} & 
	\Bun_G & 
	\Bun_G^b \ar[l]_-{i^b}  
}
\]
and morphisms $q_b \colon \cM_b \to [*/\ul{G_b(F)}]$ 
and $\wt{q}_b \colon \wt{\cM}_b \to *$
as \cite[V.3]{FaScGeomLLC}. 
Here $i_b$ and $\wt{i}_b$ are sections of $q_b$ and $\wt{q}_b$ 
respectively explained in 
\cite[Proposition V.3.6]{FaScGeomLLC}. 
Then $i^b$, $j^b$ and $\pi_b$ factor through 
\[
 i'^b \colon \Bun_G^b \to \Bun_G^{\leq b}, \quad 
 j'^b \colon \Bun_G^{<b} \to \Bun_G^{\leq b}, \quad 
 \pi'_b \colon \cM_b \to \Bun_G^{\leq b} . 
\]

\begin{lem}\label{lem:hbinv}
The functors $h_{b,\natural}$ and $Rh_{b,*}$ are quasi-inverses of 
the equivalence 
\[
h_b^* \colon D_{\lis}(\Bun_G^b,\Lambda) \to D_{\lis}([*/\ul{G_b(F)}],\Lambda) 
\] 
of categories. 
\end{lem}
\begin{proof}
Since $h_b^*$ is an equivalence of categories by 
\cite[Proposition VII.7.1]{FaScGeomLLC}, 
its left adjoint $h_{b,\natural}$ and right adjoin $h_{b,*}$ 
give quasi-inverses. 
\end{proof}

We define 
$\wt{i}_{b,!} \colon D_{\lis}(*,\Lambda) \to 
D_{\lis}(\wt{\cM}_b,\Lambda)$ and 
$\wt{i}_b^{!} \colon D_{\lis}(\wt{\cM}_b,\Lambda) \to 
D_{\lis}(*,\Lambda)$ by 
\begin{align*}
	\wt{i}_{b,!} =\cone (\wt{j}_{b,\natural} \wt{j}_b^* \to \id ) \circ \wt{q}_b^*, \quad 
	\wt{i}_b^! =R\wt{q}_{b,\lis *} \circ \fib (\id \to R\wt{j}_{b,\lis *} \wt{j}_b^* ). 
\end{align*}
Then $\wt{i}^b_!$ is a left adjoint of $\wt{i}^{b,!}$. 
We define 
$i_{b,!} \colon D_{\lis}([*/\ul{G_b(F)}],\Lambda) \to 
D_{\lis}(\cM_b,\Lambda)$ and 
$i_b^{!} \colon D_{\lis}(\cM_b,\Lambda) \to 
D_{\lis}([*/\ul{G_b(F)}],\Lambda)$ by 
\begin{align*}
i_{b,!} =\cone (j_{b,\natural} j_b^* \to \id ) \circ q_b^*, \quad 
i_b^! =Rq_{b,\lis *} \circ \fib (\id \to Rj_{b,\lis *} j_b^* ). 
\end{align*}
Then $i^b_!$ is a left adjoint of $i^{b,!}$. 
Further we define 
$i'^b_! \colon D_{\lis}(\Bun_G^b,\Lambda) \to 
D_{\lis}(\Bun_G^{\leq b},\Lambda)$ and 
$i'^{b,!} \colon D_{\lis}(\Bun_G^{\leq b},\Lambda) \to 
D_{\lis}(\Bun_G^b,\Lambda)$ by 
\begin{align*}
i'^b_! =\pi'_{b,\natural} \circ i_{b,!} \circ h_b^*, \quad 
i'^{b,!} =Rh_{b,*} \circ i_b^! \circ \pi'^*_b . 
\end{align*}
Then $i'^b_!$ is a left adjoint of $i'^{b,!}$. 
We define 
$i^b_! \colon D_{\lis}(\Bun_G^b,\Lambda) \to 
D_{\lis}(\Bun_G,\Lambda)$ and 
$i^{b,!} \colon D_{\lis}(\Bun_G,\Lambda) \to 
D_{\lis}(\Bun_G^b,\Lambda)$ by 
\[
i^b_!=j^{\leq b}_{\natural} \circ i'^b_!, \quad 
i^{b,!}=i'^{b,!} \circ j^{\leq b,*}. 
\]
Then $i'^b_!$ is a left adjoint of $i'^{b,!}$. 

\begin{lem}\label{lem:ortdec}
%\begin{enumerate}
%\item\label{en:ortdecM}
%For $A \in D_{\lis}(\cM_b,\Lambda)$, there is a distinguished triangle $A_1 \to A \to A_2 \to$ where $A_1 \in j_{b,\natural}D_{\lis}(\cM_b^{\circ},\Lambda)$ and $A_2 \in i_{b,!}D_{\lis}([*/\ul{G_b(F)}],\Lambda)$. Further, the full subcategories $j_{b,\natural}D_{\lis}(\cM_b^{\circ},\Lambda)$ and $i_{b,!}D_{\lis}([*/\ul{G_b(F)}],\Lambda)$ of $D_{\lis}(\cM_b,\Lambda)$ are equivalent to $D_{\lis}(\cM_b^{\circ},\Lambda)$ and $D_{\lis}([*/\ul{G_b(F)}],\Lambda)$ by the restrictions respectively. 
%\item\label{en:ortdecBun}
For $A \in D_{\lis}(\Bun_G^{\leq b},\Lambda)$, there is a distinguished 
triangle $A_1 \to A \to A_2 \to$ where 
$A_1 \in j^b_{\natural}D_{\lis}(\Bun_G^{< b},\Lambda)$ and 
$A_2 \in i^b_{!}D_{\lis}(\Bun_G^{b},\Lambda)$. 
Further, the full subcategories 
$j^b_{\natural}D_{\lis}(\Bun_G^{< b},\Lambda)$ and 
$i^b_{!}D_{\lis}(\Bun_G^{b},\Lambda)$ of 
$D_{\lis}(\Bun_G^{\leq b},\Lambda)$ are equivalent to 
$D_{\lis}(\Bun_G^{< b},\Lambda)$ and 
$D_{\lis}(\Bun_G^{b},\Lambda)$ by the restrictions respectively. 
%\end{enumerate}

Further similar claims hold for $\wt{\cM}_b$ and $\cM_b$. 
\end{lem}
\begin{proof}
The claim for $\Bun_G^{\leq b}$ is proved in the proof of \cite[Proposition VII.7.3]{FaScGeomLLC}. 
The claims for $\wt{\cM}_b$ and $\cM_b$ are proved in the same way. 
\end{proof}

\begin{lem}\label{lem:jjss}
	\begin{enumerate}
		\item\label{en:jjiiwtM}
		We have isomorphisms 
\[
\cone (\wt{j}_{b,\natural}\wt{j}_b^*  \to \id )  
\cong \wt{i}_{b,!} \wt{i}_b^*, \quad  
R\wt{i}_{b,\lis *} \wt{i}_b^!  \cong \fib (\id \to R\wt{j}_{b,\lis *} \wt{j}_b^* ). 
\]
		\item\label{en:jjiiM}
		We have isomorphisms 
		\[
		\cone (j_{b,\natural}j_b^*  \to \id )  
		\cong i_{b,!} i_b^*, \quad  
		Ri_{b,\lis *} i_b^!  \cong \fib (\id \to Rj_{b,\lis *}j_b^* ). 
		\]
		\item\label{en:jjiiBun}
		We have isomorphisms 
		\[\cone (j'^b_{\natural}j'^{b,*}  \to \id )  
		\cong i'^b_{!} i'^{b,*}, \quad 
		Ri'^b_{\lis *} i'^{b,!}  \cong \fib (\id \to Rj'^b_{\lis *}j'^{b,*} ). 
		\]
	\end{enumerate}
\end{lem}
\begin{proof}
Let $A \in D_{\lis}(\cM_b,\Lambda)$. 
By Lemma \ref{lem:ortdec},  
there is $A_1 \in D_{\lis}(\cM_b^{\circ},\Lambda)$ and 
$A_2 \in D_{\lis}([*/\ul{G_b(F)}],\Lambda)$ 
such that $j_{b,\natural} A_1 \to A \to i_{b,!} A_2 \to$ 
is a distinguished triangle. 
By taking $j_b^*$ and $i_b^*$, 
we have $A_1 \cong j_b^* A$ and $A_2 \cong i_b^* A$. 
Hence we obtain the first isomorphism in \ref{en:jjiiM}. 
The second isomorphism in \ref{en:jjiiM} follows from the first one 
by taking the right adjoint. 
	
The other claims are proved in the same way using 
Lemma \ref{lem:ortdec}. 
\end{proof}

%Since $s_b^* (\cone (j_{b,\natural}j_b^* A \to A)) \cong s_b^* A$, 	we have a natural morphism 	$\cone (j_{b,\natural}j_b^* A \to A) \to s_{b,*} s_b^* A$. 	We put $B=\cone \left( \cone (j_{b,\natural}j_b^* A \to A) \to s_{b,*} s_b^* A \right)$. It suffices to show the pullback of $B$ to $\wt{\cM}_{b,\bC_p^{\flat}}$ is zero. 	This follows from \cite[Proposition VII.6.7]{FaScGeomLLC} 	since $j_b^* B=s_b^* B=0$ and the assumption of \cite[Proposition VII.6.7]{FaScGeomLLC} is satisfied for $\wt{\cM}_{b,\bC_p^{\flat}}$ 	(\cf \cite[VII.7]{FaScGeomLLC}). 

\begin{comment}
	First we show that the natural morphism 
	$(s_{b,*} \Lambda )\otimes A\to s_{b,*} s_b^* A$ is an isomorphism. 
	We put $B=\cone ((s_{b,*} \Lambda )\otimes A \to s_{b,*} s_b^* A)$. 
	
	Since $\cone (j_{b,\natural}j_b^* A \to A) \cong \cone (j_{b,\natural}j_b^* \Lambda \to \Lambda) \otimes A$ and 
	$s_{b,*} s_b^* A \cong (s_{b,*} \Lambda )\otimes A$, we may assume that 
	$A=\Lambda$. Further the claim is reduced to the case where $\Lambda$ is torsion by Lemma \ref{lem:limtor}, 
	since $j_b$ is quasi-compact, separated by \cite[Proposition V.3.5 and Proposition V.3.6]{FaScGeomLLC}. 
	In the torsion case claim is trivial since $s_{b,*}$ is understood as 
	a pushforward functor $D_{\textrm{\'et}}([*/\underline{G_b(F)}],\Lambda) \to D_{\textrm{\'et}}(\cM_b,\Lambda)$ by \cite[Proposition VII.6.6]{FaScGeomLLC}. 
\end{comment}

\begin{lem}\label{lem:ibproj}
	For $A \in D_{\lis}(\cM_b,\Lambda)$ and $B \in D_{\lis}([*/\ul{G_b(F)}],\Lambda)$, we have an isomorphism 
	$i_{b,!}(i_b^* (A) \otimes^{\bL} B)\cong A \otimes^{\bL} i_{b,!} (B)$. 
\end{lem}
\begin{proof}
	We have 
	\begin{align*}
		A \otimes^{\bL} i_{b,!} (B) \cong 
		\cone (j_{b,\natural}j_b^* (A \otimes^{\bL} q_b^* B) \to A \otimes^{\bL} q_b^* B) 
		\cong i_{b,!}i_b^* (A \otimes^{\bL} q_b^* B) 
		\cong i_{b,!}(i_b^* (A) \otimes^{\bL} B), 
	\end{align*}
	where we use Lemma \ref{lem:jjss} \ref{en:jjiiM} 
	at the second isomorphism. 
\end{proof}

\begin{lem}\label{lem:fibq!}
We have $\fib (D_{q_b} \to Rj_{b,\lis*}j_b^* D_{q_b}) \cong 
Ri_{b,\lis*} \Lambda$. 
\end{lem}
\begin{proof}
By the change of coefficient and the inverse limit, 
we may assume that $\Lambda$ is torsion. Then we have 
$\fib (D_{q_b} \to Rj_{b,\lis*}j_b^* D_{q_b}) \cong Ri_{b,\lis*} i_b^! q_b^! \Lambda \cong  
Ri_{b,\lis*} \Lambda$. 
\end{proof}

\begin{lem}\label{lem:Di_b}
We have $\bD \circ i_{b,!} =i_{b,*} \circ \bD$ and 
$i_{b}^! \circ \bD = \bD \circ i_{b}^*$. 
\end{lem}
\begin{proof}
Let $A \in D_{\lis}(\cM_b,\Lambda)$. We have 
\begin{align*}
 (i_{b}^! \circ \bD) (A)=i_{b}^! R\sHom_{\lis} (A,D_{q_b}) \cong 
 Rq_{b,\lis*}(R\sHom_{\lis} (A,\fib (D_{q_b} \to Rj_{b,\lis*}j_b^* D_{q_b} )))
\end{align*}
By Lemma \ref{lem:fibq!}, this is isomorphic to 
\begin{align*}
 Rq_{b,\lis*}(R\sHom_{\lis} (A,i_{b,\lis*} \Lambda)) 
 &\cong 
 Rq_{b,\lis*}(Ri_{b,\lis*} R\sHom_{\lis} (i_b^* A,\Lambda)) \\ 
 &\cong R\sHom_{\lis} (i_b^* A,\Lambda)) = 
 (\bD \circ i_b^*) (A). 
\end{align*}
Hence we have $i_{b}^! \circ \bD = \bD \circ i_{b}^*$. 
Another claim follows from this by adjoint. 
\end{proof}

The following lemma is already known (\cf \cite[IX.3]{FaScGeomLLC}). 

\begin{lem}\label{lem:Di^b}
We have $\bD \circ i^b_{!} =i^b_{*} \circ \bD$ and 
	$i^{b,!} \circ \bD = \bD \circ i^{b,*}$. 
\end{lem}
\begin{proof}
Let $A \in D_{\lis}(\Bun_G,\Lambda)$. 
We have 
\begin{align*}
h_b^* ((i^{b,!} \circ \bD)(A)) &\cong i_b^! \pi_b^* R\sHom_{\lis} (A,D_{\Bun_G}) \cong i_b^! R\sHom_{\lis} (\pi_b^*A,\pi_b^* D_{\Bun_G}) \\ 
&\cong i_b^! \bD ((\pi_b^*A) \otimes D_{\pi_b}) 
\cong \bD (i_b^* \pi_b^*A) \otimes i_b^* D_{\pi_b}^{-1} , 
\end{align*}
where we use Lemma \ref{lem:ibproj} at the second isomorphism and 
Lemma \ref{lem:Di_b} at the fourth isomorphism. 
On the other hand we have 
\begin{align*}
h_b^* ((i^b_{*} \circ \bD)(A)) \cong h_b^* R\sHom_{\lis} 
(i^{b,*} A, D_{\Bun_G^b}) &\cong 
R\sHom_{\lis} (h_b^* i^{b,*} A, h_b^* D_{\Bun_G^b})\\  
&\cong \bD (i_b^* \pi_b^*A) \otimes D_{h_b}^{-1}, 
\end{align*}
where we use Lemma \ref{lem:ibproj} at the second isomorphism. 
Hence $i^{b,!} \circ \bD = \bD \circ i^{b,*}$ follows from 
\cite[Proposition 23.12]{SchEtdia}. 
Another claim follows from this by adjoint. 
\end{proof}

\begin{comment}
\begin{lem}
Let $D \in D_{\lis}(\cM_b,\Lambda)$ be an invertible object. 
For $A \in D_{\lis}(\cM_b^{\circ},\Lambda)$, we have 
\[
 j_{b,\natural} R\sHom_{\lis} (A,j_b^* D) \cong 
 R\sHom_{\lis} (j_{b,\lis*} A,D). 
\]
\end{lem}
\begin{proof}
By the projection formula, we may assume that $D=\Lambda$. 
By \cite[Proposition VII.7.7]{FaScGeomLLC}, 
we have 
\[
 j^b_{\natural}
\]
\end{proof}
\end{comment}

\begin{comment}
\begin{lem}
We have $\bD^2 (i_{b,!} \Lambda) \cong (i_{b,!} \Lambda)$. 
\end{lem}
\begin{proof}
Since $i_{b,!} \Lambda=\cone (j_{b,\natural} \Lambda \to \Lambda)$, it suffices to show that 
\[
 R\sHom_{\lis} (R\sHom_{\lis} (j_{b,\natural} \Lambda,\Lambda) ) \cong j_{b,\natural} \Lambda. 
\]
We have 
\[
R\sHom_{\lis} (j_{b,\natural} \Lambda,\Lambda) \cong 
j_{b,\lis *} \Lambda 
\]
by \cite[Proposition VII.3.1(i)]{FaScGeomLLC}. 
Hence it suffices to show 
\[
 R\sHom_{\lis} (j_{b,\lis *} \Lambda,\Lambda) \cong j_{b,\natural} \Lambda. 
\]
This follows from \cite[Proposition VII.7.7]{FaScGeomLLC} 
by taking the pullback under $\pi_b$ using Lemma \ref{lem:lisbc} 
and Lemma \ref{lem:pullhomlis}. 
\end{proof}
\end{comment}

\begin{lem}\label{lem:i^!des}
We have 
$\wt{i}_b^! \cong \wt{i}_b^* (\fib (\id \to R\tilde{j}_{b,\lis *}\tilde{j}_b^*))$, 
$i_b^! \cong i_b^* (\fib (\id \to Rj_{b,\lis *}j_b^*))$ and 
$i'^{b,!} \cong i'^{b,*} (\fib (\id \to Rj'^b_{\lis *}j'^{b,*}))$. 
\end{lem}
\begin{proof}
For $A \in D_{\lis}(\cM_b,\Lambda)$ and 
$B \in D_{\lis}([*/\ul{G_b(F)]},\Lambda)$, we have 
\[
 \Hom (i_{b,!}(B),A) \cong \Hom (i_{b,!}(B),\fib (A \to Rj_{b,\lis *}j_b^* A)) \cong 
 \Hom (B,i_b^* (\fib (A \to Rj_{b,\lis *}j_b^* A))) 
\]
by Lemma \ref{lem:ortdec}. Hence we obtain the second claim. 
Other claims are proved similarly. 
\end{proof}

\begin{lem}\label{lem:i^!i_!}
We have 
$\wt{i}_b^! \wt{i}_{b,!} \cong \id$, 
$i_b^! i_{b,!} \cong \id$ and 
$i'^{b,!} i'^b_! \cong \id$. 
\end{lem}
\begin{proof}
We can check these using Lemma \ref{lem:i^!des}. 
\end{proof}

\begin{lem}
We have 
$\wt{i}_{b,!} \cong R\wt{i}_{b,\lis *}$, 
$i_{b,!} \cong Ri_{b,\lis *}$ and 
$i'^b_{!} \cong Ri'^b_{\lis *}$. 
\end{lem}
\begin{proof}
By Lemma \ref{lem:jjss}, Lemma \ref{lem:i^!des} 
and Lemma \ref{lem:i^!i_!}, 
we have 
\begin{align*}
i_b^* Ri_{b,\lis *} \cong 
i_b^* Ri_{b,\lis *} i_b^! i_{b,!} \cong i_b^*  \fib (\id \to Rj_{b,\lis *}j_b^*) i_{b,!}   
\cong i_b^! i_{b,!} \cong \id. 
\end{align*}
Hence $i_{b,!} \cong i_{b,\lis *}$ follows from 
Lemma \ref{lem:ortdec} using Lemma \ref{lem:lisbc}. 
Other claims are proved similarly. 
\end{proof}

For a compact open subgroup $K$ of $G_b(F)$, 
we consider the fiber products 
\[
  \xymatrix{
  \Sht_{G,b,K,b',\bC_p^{\flat}}^{\mu_{\bullet}} \ar[rr]^-{f_K}  \ar[dd] & & 
  \mathrm{Hck}_{b'}^{\mu_{\bullet}} \ar[r]^-{f_{b'}} \ar[d]& 
  \Spa \bC_p^{\flat} \ar[d]^-{t_{b'}} \\ 
  %\mathrm{Hck}_{G,K}^{\mu_{\bullet}} \ar[rr] \ar[d] 
  & & 
  \mathrm{Hck}^{\mu_{\bullet}} \ar[r]^-{p_{2,X}} \ar[d]^-{p_1} & 
  \Bun_G \times \mathrm{Div}_X^m \\ 
  [*/K] \ar[r]^-{h_K} &  
  \Bun_G^b \ar[r]^-{i^b} & \Bun_G & 
 }
\]
where $h_K$ and $t_{b'}$ are the compositions 
\begin{align*}
 &[*/K] \xrightarrow{h_{K,G_b(F)}} [*/G_b(F)] 
\stackrel{h_b}{\longrightarrow} \Bun_G^b , \\ 
 & 
 \Spa \bC_p^{\flat} \longrightarrow \Bun_G^{b'} \times \mathrm{Div}_X^m \longrightarrow \Bun_G \times \mathrm{Div}_X^m
\end{align*}
of the natural morphisms. 
Let $p_{1,b'} \colon \mathrm{Hck}_{b'}^{\mu_{\bullet}} \to \mathrm{Hck}^{\mu_{\bullet}} \stackrel{p_1}{\to} \Bun_G$. 
We put 
\[
 f_{K,!} \Lambda = p_{1,b'}^* i^b_! h_{K,!} \Lambda. 
\]
\begin{rem}
If $b$ is basic, $f_K$ is etale, in particular $\ell$-cohomologically smooth. In this case, the above definition of $f_{K,!} \Lambda$ coincides 
with the general definition before. 
\end{rem}
We put 
\[
 R\Gamma_{\mathrm{c}} 
 (\Sht_{G,b,K,b'}^{\mu_{\bullet}}) = f_{b',\natural} \bigl( (f_{K,!} \Lambda) \otimes^{\bL} \mathrm{IC}'_{\mu_{\bullet}}\bigr). 
\] 
We can view 
\[
 R\Gamma_{\mathrm{c}} 
 (\Sht_{G,b,K,b'}^{\mu_{\bullet}}) \cong 
 t_{b'}^* T_{\mu_{\bullet}}(i^b_! h_{K,!} \Lambda) 
\]
as an object of $D(G_b(F) \times W_E)$ by 
\cite[Corollary IX.2.3]{FaScGeomLLC}. 
For a compact open subgroup $K'$ of $G_{b'}(F)$, we define 
$R\Gamma_{\mathrm{c}} (\Sht_{G,b,b',K'}^{\mu_{\bullet}})$ in the symmetric way. 
Since $\mathrm{IC}_{\mu_{\bullet}}$ and $\mathrm{IC}_{-\mu_{\bullet}}$ 
corresponds 
under the natural isomorphism 
$\Sht_{G,b,b'}^{\mu_{\bullet}} \simeq \Sht_{G,b',b}^{-\mu_{\bullet}}$, 
we have 
\begin{equation*}
 R\Gamma_{\mathrm{c}} (\Sht_{G,b,b',K'}^{\mu_{\bullet}}) \cong 
 t_{b}^* T_{-\mu_{\bullet}}(i^{b'}_! h_{K',!} \Lambda). 
\end{equation*}

\begin{rem}
If $b$ is basic, $R\Gamma_{\mathrm{c}} 
(\Sht_{G,b,K,b'}^{\mu_{\bullet}})$ is identified with 
$(f_{b'}\circ f_K)_{\natural} (\mathrm{IC}'_{\mu_{\bullet}})$. 
We define $R\Gamma_{\mathrm{c}} 
(\Sht_{G,b,K,b'}^{\mu_{\bullet}})$ as above since we do not have a good definition of \[
f_{K,!} \colon D_{\lis}(\Sht_{G,b,K,b',\bC_p^{\flat}}^{\mu_{\bullet}},\Lambda) \to D_{\lis}(\mathrm{Hck}_{b'}^{\mu_{\bullet}},\Lambda) 
\]
for a general $b$. 
\end{rem}

We put 
\[
R\Gamma_{\mathrm{c}} 
(\Sht_{G,b,b'}^{\mu_{\bullet}})=
 \varinjlim_{K \subset G_b(F)} R\Gamma_{\mathrm{c}} 
(\Sht_{G,b,K,b'}^{\mu_{\bullet}}) . 
\]

\begin{lem}\label{lem:qiid}
%For $A \in D_{\lis}(\cM_b,\Lambda)$, we have an isomorphism $q_{b,!}(\cone (j_{b,\natural}j_b^* A \to A)) \cong s_b^* A$. 
We have $q_{b,!} \circ i_{b,!}=\id$. 
\end{lem}
\begin{proof}
Let $B \in D_{\lis}([*/\ul{G_b(F)}],\Lambda)$. 
Then we have 
$i_{b,!}(B)\cong \cone (j_{b,\natural}j_b^* \Lambda \to \Lambda) \otimes^{\bL} q_b^*B$. 
%\begin{align*}		\cone (j_{b,\natural}j_b^* A \to A) \cong s_{b,*} s_b^* A 		&\cong \cone (j_{b,\natural}j_b^* (q_b^* s_b^* A) \to q_b^* s_b^* A) \\		&\cong \cone (j_{b,\natural}j_b^* \Lambda \to \Lambda)\otimes^{\bL} q_b^* s_b^* A. 	\end{align*}
Hence we have 
\[
 (q_{b,!} \circ i_{b,!})(B) \cong 
	q_{b,\natural}(\cone (j_{b,\natural}j_b^* \Lambda \to \Lambda) \otimes D_{q_b}^{-1}) \otimes^{\bL} B. 
\]
It remains to show $q_{b,\natural}(\cone (j_{b,\natural}j_b^* \Lambda \to \Lambda) \otimes D_{q_b}^{-1}) \cong \Lambda$. 
It suffices to show this 
after taking a pullback via $\Spa \bC_p^{\flat} \to [*/\ul{G_b(F)}]$
since the induced actions of $G_b(F)$ on the both sides are trivial. 
Let $j_U \colon U \to \wt{\cM}_{b,\bC_p^{\flat}}$ 
be a quasicompact open neighborhood of 
$\wt{i}_b (\Spa \bC_p^{\flat})$. 
We have 
\[
 \wt{q}_{b,\natural}(\cone (\wt{j}_{b,\natural}\wt{j}_b^* \Lambda \to \Lambda) \otimes (\wt{q}_b^! \Lambda)^{-1}) \cong 
 (\wt{q}_{b} \circ j_U )_{\natural} j_U^* (R\wt{i}_{b,\lis *}(\Lambda) \otimes D_{\wt{q}_b}^{-1}). 
\]
Then the question is reduced to the torsion case by 
by Lemma \ref{lem:limtor}, 
since $\wt{q}_{b} \circ j_U$ is quasi-compact, separated 
by \cite[Proposition V.3.5]{FaScGeomLLC}. 
In the torsion case, the claim follows from \cite[Proposition VII.5.2]{FaScGeomLLC} and $\cone (j_{b,\natural}j_b^* \Lambda \to \Lambda) \cong i_{b,!}(\Lambda)$. 
\end{proof}

%\begin{lem}\label{lem:i!proj} For $A \in D_{\lis}(\cM_b,\Lambda)$ and $B \in D_{\lis}([*/\ul{G_b(F)}],\Lambda)$, we have an isomorphism $i_{b,!} (i_b^* A \otimes^{\bL} B) \cong A \otimes^{\bL} i_{b,!} B$. \end{lem}

\begin{lem}\label{lem:qAiB}
For $A \in D_{\lis}(\cM_b,\Lambda)$ and $B \in D_{\lis}([*/\ul{G_b(F)}],\Lambda)$, we have an isomorphism $q_{b,\natural}(A \otimes^{\bL} i_{b,!} B) \cong i_b^* (A \otimes D_{q_b} ) \otimes^{\bL} B$. 
\end{lem}
\begin{proof}
We have 
\begin{align*}
A \otimes^{\bL} i_{b,!} B &\cong 
\cone (j_{b,\natural} j_b^* A \to A) \otimes^{\bL} q_b^* B \\
&\cong \cone (j_{b,\natural} j_b^* (A \otimes D_{q_b}) \to A \otimes D_{q_b}) \otimes D_{q_b}^{-1} \otimes^{\bL} q_b^* B \\ 
&\cong(i_{b,!}i_b^* (A \otimes D_{q_b})) \otimes D_{q_b}^{-1} \otimes^{\bL} q_b^* B, 
\end{align*}
where we use Lemma \ref{lem:jjss} \ref{en:jjiiM} 
at the last isomorphism. 
Hence we have 
\begin{align*}
q_{b,\natural} (A \otimes^{\bL} i_{b,!} B) &\cong 
q_{b,\natural} ((i_{b,!}i_b^* (A \otimes D_{q_b})) \otimes D_{q_b}^{-1}) \otimes^{\bL} B \\ 
&\cong q_{b,!} (i_{b,!}i_b^* (A \otimes D_{q_b})) \otimes^{\bL} B 
\cong i_b^* (A \otimes D_{q_b} ) \otimes^{\bL} B, 
\end{align*}
where we use Lemma \ref{lem:qiid} at the last isomorphism. 
\end{proof}

\begin{lem}\label{lem:BunAiB}
	Let $A \in D_{\lis}(\Bun_G,\Lambda)$ and $B \in D_{\lis}(\Bun_G^b,\Lambda)$. Then we have 
	\[
	R\Gamma_{\natural}(\Bun_G, A \otimes^{\bL} i^b_! B) \cong 
	R\Gamma_{\natural}([*/\ul{G_b(F)}], h_b^* (i^{b,*} A \otimes^{\bL} B) \otimes^{\bL} i_b^* D_{q_b} ). 
	\]
\end{lem}
\begin{proof}
We have 
\begin{align*}
R\Gamma_{\natural}(\Bun_G, A \otimes^{\bL} i^b_! B) 
&\cong
R\Gamma_{\natural}(\cM_b, \pi_b^* A \otimes^{\bL} i_{b,!} h_b^* B )\\
%&\congR\Gamma_{\natural}(\cM_b, \pi_b^* A \otimes^{\bL} \cone (j_{b,\natural} j_b^* q_b^* h_b^* B \to q_b^* h_b^* B) )\\ 
%&\cong R\Gamma_{\natural}(\cM_b, \cone (j_{b,\natural} j_b^* \pi_b^* A \to \pi_b^* A ) \otimes^{\bL}  q_b^* h_b^* B )\\ 
&\cong R\Gamma_{\natural}([*/\ul{G_b(F)}], i_b^* (\pi_b^* A \otimes D_{q_b}) \otimes^{\bL}  h_b^* B )\\ 
& \cong 
R\Gamma_{\natural}([*/\ul{G_b(F)}], h_b^* (i^{b,*} A \otimes^{\bL} B) \otimes^{\bL} i_b^* D_{q_b} ),
\end{align*}
where we use Lemma \ref{lem:qAiB} 
at the second isomorphism. 
\end{proof}

\begin{lem}\label{lem:cohsym}
We have a natural isomorphism 
\[
 \varinjlim_{K \subset G_b(F)} R\Gamma_{\mathrm{c}} 
 (\Sht_{G,b,K,b'}^{\mu_{\bullet}}) \cong 
 \varinjlim_{K' \subset G_{b'}(F)} R\Gamma_{\mathrm{c}} 
 (\Sht_{G,b,b',K'}^{\mu_{\bullet}}) . 
\]
\end{lem}
\begin{proof}
It suffices to show that 
\[
R\Gamma_{\mathrm{c}} 
(\Sht_{G,b,K,b'}^{\mu_{\bullet}})_{K'} \cong R\Gamma_{\mathrm{c}} 
(\Sht_{G,b,b',K'}^{\mu_{\bullet}})_{K} 
\]
for enough small $K$ and $K'$. 
We consider the following diagram: 
\[
\xymatrix{
	[\Spa \bC_p^{\flat}/\ul{G_b(F)}] \ar[d]^-{h_b} 
	 & [\Spa \bC_p^{\flat}/\ul{K}] \ar[l]_-{h_{K,b}} \ar[ld]_-{h_K} & 
	 & [\Spa \bC_p^{\flat}/\ul{K'}] \ar[r]^-{h_{K',b'}} \ar[rd]^-{h_{K'}} & 
	[\Spa \bC_p^{\flat}/\ul{G_{b'}(F)}] \ar[d]_-{h_{b'}} \\ 
	 \Bun_{G,\bC_p^{\flat}}^b \ar[r]^-{i^b} 
	 & \Bun_{G,\bC_p^{\flat}}
 	 & \mathrm{Hck}^{\mu_{\bullet}}_{\bC_p^{\flat}} \ar[r]^-{p_2} \ar[l]_-{p_1} 
 	 & \Bun_{G,\bC_p^{\flat}} 
	 & \Bun^{b'}_{G,\bC_p^{\flat}} \ar[l]_-{i^{b'}} 
}
\]
We have 
\begin{align}
&R\Gamma_{\natural ,\bC_p^{\flat}}(\mathrm{Hck}^{\mu_{\bullet}}_{\bC_p^{\flat}}, p_1^* i^b_! h_{K,!} \Lambda \otimes^{\bL} p_2^* i^{b'}_! h_{K',!}\Lambda \otimes^{\bL} \IC'_{\mu_{\bullet}} ) \label{eq:RGamHck} \\ 
&\cong R\Gamma_{\natural ,\bC_p^{\flat}}(\Bun_{G,\bC_p^{\flat}}, T_{\mu_{\bullet}}(i^b_! h_{K,!} \Lambda) \otimes^{\bL} i^{b'}_! h_{K',!}\Lambda ) \notag \\
&\cong 
R\Gamma_{\natural ,\bC_p^{\flat}}([\Spa \bC_p^{\flat}/\ul{G_{b'}(F)}], h_{b'}^* (i^{b,*} T_{\mu_{\bullet}}(i^b_! h_{K,!} \Lambda) \otimes^{\bL} h_{K',!}\Lambda ) \otimes s_{b'}^* D_{q_{b'}} ), \label{eq:RGamb'}
\end{align}
where we use Lemma \ref{lem:BunAiB} at the last isomorphism. 
We have 
\begin{align*}
h_{b'}^*h_{K',!}\Lambda \otimes s_{b'}^* D_{q_{b'}} 
&\cong h_{b'}^* h_{b',\natural} ((h_{K',b',\natural}\Lambda) \otimes D_{h_{b'}}^{-1} ) \otimes s_{b'}^* D_{q_{b'}} \\ 
&\cong h_{K',b',\natural} h_{K',b'}^*(D_{h_{b'}}^{-1} \otimes s_{b'}^* D_{q_{b'}}), 
\end{align*}
where we use Lemma \ref{lem:hbinv} at the last isomorphism. 
Since $h_{b'}$ and $q_{b'}$ are cohomologically smooth of the same dimension,  $D_{h_{b'}}^{-1} \otimes s_{b'}^* D_{q_{b'}}$ is etale locally trivial. 
Hence we may assume that $K'$ is small enough so that 
$h_{K',b'}^*(D_{h_{b'}}^{-1} \otimes s_{b'}^* D_{q_{b'}}) \cong \Lambda$. Then \eqref{eq:RGamb'} is isomorphic to 
\[
R\Gamma_{\natural ,\bC_p^{\flat}} \bigl( [\Spa \bC_p^{\flat}/\ul{G_{b'}(F)}], h_{b'}^* i^{b',*} T_{\mu_{\bullet}}(i^b_! h_{K,!} \Lambda) \otimes^{\bL} h_{K',b',\natural} \Lambda \bigr) \cong 
R\Gamma_{\mathrm{c}} (\Sht_{G,b,K,b'}^{\mu_{\bullet}})_{K'} . 
\]
Since \eqref{eq:RGamHck} is symmetric with respect to 
$(b,K)$ and $(b',K')$, the claim follows. 
\end{proof}

\begin{prop}
	We have  
	$R\Gamma_{\mathrm{c}} (\Sht_{G,b,b'}^{\mu_{\bullet}}) \cong 
	R\Gamma_{\mathrm{c}} (\Sht_{G,b',b}^{-\mu_{\bullet}})$. 
\end{prop}
\begin{proof}
This follows from Lemma \ref{lem:cohsym}. 
\end{proof}

\begin{prop}
\begin{enumerate}
\item 
If $K$ is pro-$p$, then $R\Gamma_{\mathrm{c}} (\Sht_{G,b,K,b'}^{\mu_{\bullet}})$ is a compact object in 
$D(G_{b'}(F),\Lambda)$. 
\item 
For $i \in \bZ$, 
$H^i_{\mathrm{c}} (\Sht_{G,b,K,b'}^{\mu_{\bullet}})$ is finitely generated smooth $G_{b'}(F)$-representation. 
\item 
If $\rho$ is an admissible representation of $G_{b'}$ over $\Lambda$, then $R\Hom_{G_{b'}}(R\Gamma_{\mathrm{c}} (\Sht_{G,b,K,b'}^{\mu_{\bullet}}),\rho)$ is a perfect complex of $\Lambda$-modules. 
\item 
If $\Lambda=\ol{\bQ}_{\ell}$ and 
$\rho$ is a finite length representation of $G_{b'}(F)$ over $\ol{\bQ}_{\ell}$, then 
\[
\varinjlim_{K \subset G_b(F)} R^i\Hom_{G_{b'}(F)}(R\Gamma_{\mathrm{c}} (\Sht_{G,b,K,b'}^{\mu_{\bullet}}),\rho) 
\] 
is finite length representatin of $G_b(F)$ for $i \in \bZ$.  
\end{enumerate}
\end{prop}
\begin{proof}
We have 
\[
 \varinjlim_{K \subset G_b(F)} R\Hom_{G_{b'}}(R\Gamma_{\mathrm{c}} (\Sht_{G,b,K,b'}^{\mu_{\bullet}}),\rho) 
 \cong 
 i^{b,!} T_{\mu_{\bullet}^{\vee}} Ri^{b'}_{\lis *} Rh_{b',*} [\rho] . 
\]
Then the claims are proved in the same way as \cite[IX.3]{FaScGeomLLC} 
using Lemma \ref{lem:Di^b}. 
\end{proof}

We put 
\[
 H_{\mathrm{c}}^* 
 (\Sht_{G,b,b'}^{\mu_{\bullet}}) = \sum_{i \in \bZ} (-1)^i 
 R^i\Gamma_{\mathrm{c}}
 ( \Sht_{G,b,b'}^{\mu_{\bullet}}). 
\]

%Let $U$ be the unipotent radical of $B$. By Lemma \ref{lem:Jbnat}, we have isomorphisms 
%\[ H_{\mathrm{c}}^* (\Sht_{G,b,b'}^{\mu}) =  H_{\mathrm{c}}^* (\Sht_{G,b,b'}^{\mu}/\wt{J}_b^{>0})  =H_{\mathrm{c}}^* (\Sht_{G,b,b'}^{\mu}/\wt{J}_{b'}^{>0}) \]
%and these have actions of $G_b(F)=(\wt{J}_b/\wt{J}_b^{>0})(\bC_p^{\flat})$ and $G_{b'}(F)=(\wt{J}_{b'}/\wt{J}_{b'}^{>0})(\bC_p^{\flat})$. 

\section{Convolution morphism and twist morphism}\label{sec:conv}

\subsection{Convolution morphism}

Let $\Delta_{m,\Spd F}$ denote the diagonal subspace of 
$(\Spd F)^m$. 
For $1 \leq i < j \leq m$, let 
$\mathrm{pr}_{i,j} \colon (\Spd F)^m \to (\Spd F)^2$ 
denote the projection to the $(i,j)$-component. 
We put 
\[
U_m = (\Spd F)^m \setminus \bigcup_{1 \leq i < j \leq m} 
\mathrm{pr}_{i,j}^{-1} \left( 
\bigcup_{n \in \bZ \setminus \{ 0\}} (\varphi \times 1)^n 
(\Delta_{2,\Spd F}) \right). 
\]
This is an open subspace of $(\Spd F)^m$ 
which contains $\Delta_{m,\Spd F}$. 

Let $b_0 ,\ldots , b_m \in G(\breve{F})$ and 
$\mu_{\bullet} =(\mu_1 , \ldots , \mu_m)$ 
where $\mu_i \in X_* (T)$ 
for $1 \leq i \leq m$. 
We put 
\[
\Sht_{G,b_0,b_m,U_m}^{\mu_{\bullet}} = 
\Sht_{G,b_0,b_m}^{\mu_{\bullet}} 
\times_{(\Spd F)^m} U_m. 
\]
We define the convolution morphism 
\[
m_{b_{\bullet},\mu_{\bullet},U_m} \colon 
(\Sht_{G,b_0,b_1}^{\mu_1}
\times \cdots \times 
\Sht_{G,b_{m-1},b_m}^{\mu_m}) \times_{(\Spd F)^m} U_m \to 
\Sht_{G,b_0,b_m,U_m}^{\mu_{\bullet}} 
\]
over $\Spd \breve{E_1} \times \cdots \times \Spd \breve{E_m}$ 
as follows. 
Let $S =\Spa (R,R^+) \in \Perf_{\ol{\bF}_q}$ and 
\[
(S_i^{\sharp}, 
\cP_i, \varphi_{\cP_i}, \iota_{(0,r],i}, \iota_{[r',\infty],i} 
)_{1 \leq i \leq m}
\]
be objects giving an $S$-valued point of 
\[
(\Sht_{G,b_0,b_1}^{\mu_1}
\times \cdots \times 
\Sht_{G,b_{m-1},b_m}^{\mu_m}) \times_{(\Spd F)^m} U_m. 
\] 
Define $\cP$ by gluing 
$\cP_1|_{\cY_{(0,r]}(S)}$ and 
$\cP_m|_{\cY_{[r',\infty)}(S)}$ by the following modifications: 
\begin{itemize}
	\item 
	Modifications occur only at 
	$\bigcup_{i=1}^m \bigcup_{n \geq 0} \varphi^{-n} (S_i^{\sharp})$. 
	\item
	Take $1 \leq i_0 \leq m$. 
	Put 
	\[
	I_{i_0} = \{ 1 \leq i \leq m \mid 
	S_i^{\sharp} =S_{i_0}^{\sharp} \}. 
	\]
	Define the modification at $S_{i_0}^{\sharp}$ 
	by the composite of the modifications at 
	$S_{i_0}^{\sharp}$ given by $\varphi_{\cP_i}$ 
	for all $i \in I_{i_0}$. 
	For $n >0$, 
	the modification at $\varphi^{-n} (S_{i_0}^{\sharp})$ 
	is given by the pullback under $\varphi^n$ 
	of the modification at $S_{i_0}^{\sharp}$. 
\end{itemize}
Then $\cP$ is naturally equipped with an isomorphism 
\[
\varphi_{\cP} \colon (\varphi_S^* \cP )
|_{``S \times \Spa F" \setminus \bigcup_{i=1}^m S_i^{\sharp}} 
\simeq 
\cP|_{``S \times \Spa F" \setminus \bigcup_{i=1}^m S_i^{\sharp}} . 
\]
Further, we have 
isomorphisms 
\begin{align*}
	\cP|_{\cY_{(0,r]}(S)} &= \cP_1|_{\cY_{(0,r]}(S)} 
	\xrightarrow{\iota_{(0,r],1}} G \times \cY_{(0,r]}(S) ,\\ 
	\cP|_{\cY_{[r',\infty)}(S)} &= \cP_m|_{\cY_{[r',\infty)}(S)} 
	\xrightarrow{\iota_{[r',\infty),m}} 
	G \times \cY_{[r',\infty)}(S). 
\end{align*}
These gives 
an $S$-valued point of 
$\Sht_{G,b_0,b_m,U_m}^{\mu_{\bullet}}$. 
Thus we obtain $m_{b_{\bullet},\mu_{\bullet},U_m}$. 

We define 
\[
\Gr_{G,\Spd E_1 \times \cdots \times \Spd E_m,\leq \mu_{\bullet}}, \quad  
\wt{\Gr}_{G,\Spd E_1 \times \cdots \times \Spd E_m,\leq \mu_{\bullet}} 
\]
as in \cite[Definition 20.4.4]{ScWeBLp}. 
Then we have a convolution morphism 
\[
m_{\mu_{\bullet}} \colon 
\wt{\Gr}_{G,\Spd E_1 \times \cdots \times \Spd E_m,\leq \mu_{\bullet}} 
\lra 
\Gr_{G,\Spd E_1 \times \cdots \times \Spd E_m,\leq \mu_{\bullet}} 
\]
by \cite[Proposition 20.4.5]{ScWeBLp}. 
Note that 
\[
\Gr_{G,\Spd E_1 \times \cdots \times \Spd E_m,\leq \mu_{\bullet}} 
\times_{(\Spd F)^m} U_m \simeq 
\Gr_{G,\Spd E_1 \times \cdots \times \Spd E_m, 
	\leq \mu_{\bullet}}^{\mathrm{tw}} \times_{(\Spd F)^m} U_m . 
\]
Then we have a morphism 
\[
\Sht_{G,b_0,b_1}^{\mu_1}
\times \cdots \times 
\Sht_{G,b_{m-1},b_m}^{\mu_m} \lra 
\wt{\Gr}_{G,\Spd \breve{E}_1 \times \cdots \times \Spd \breve{E}_m,
	\leq \mu_{\bullet}}
\]
by looking at a modification at each $S_i^{\sharp}$. 
Then we have the commutative diagram 
\[
\xymatrix{
	(\Sht_{G,b_0,b_1}^{\mu_1}
	\times \cdots \times 
	\Sht_{G,b_{m-1},b_m}^{\mu_m}) \times_{(\Spd F)^m} U_m 
	\ar[r]^-{m_{b_{\bullet},\mu_{\bullet},U_m}} \ar[d] & 
	\Sht_{G,b_0,b_m,U_m}^{\mu_{\bullet}}  \ar[d] \\ 
	\wt{\Gr}_{G,\Spd \breve{E}_1 \times \cdots \times \Spd \breve{E}_m,
		\leq \mu_{\bullet}} 
	\times_{(\Spd F)^m} U_m 
	\ar[r] & 
	\Gr_{G,\Spd \breve{E}_1 \times \cdots \times \Spd \breve{E}_m, 
		\leq \mu_{\bullet}} 
	\times_{(\Spd F)^m} U_m 
}
\]
where the bottom morphism is induced by 
$m_{\mu_{\bullet}}$. 

\subsection{Twist morphism}

Let $Z^0$ be the identity component of the center of $G$. 
Let $a,a' \in Z^0(\breve{F})$ and $\lambda \in X_*(Z^0)$. 
Let $E$ be a finite extension of $F$ in $\bC_p$ 
containing 
the fields of definition of $\mu$ and $\lambda$. 
We define the morphism 
\[
 t_{b,b',a,a'}^{\mu, \lambda} \colon 
 \Sht_{G,b,b',\Spd \breve{E}}^{\mu} \times_{\Spd \breve{E}} 
 \Sht_{Z^0,a,a',\Spd \breve{E}}^{\lambda} \lra 
 \Sht_{G,ab,a'b',\Spd \breve{E}}^{\mu-\lambda} 
\]
as follows. 
Let $(S^{\sharp},\sE_{b} \to \sE_{b'})$ and 
$(S^{\sharp},\sE_{a} \to \sE_{a'})$ be modifications defining 
points in 
$\Sht_{G,b,b'}^{\mu}$ and 
$\Sht_{Z^0,a,a'}^{\lambda}$. 
Then the diagonal arrow in the diagram 
\[
 \xymatrix{
 \sE_b \times^{Z^0} \sE_{a'} \ar[r]  \ar[r] \ar[rd] & 
 \sE_{b'} \times^{Z^0} \sE_{a'}  \\ 
 \sE_b \times^{Z^0} \sE_a \ar[u] \ar[r] & 
 \sE_{b'} \times^{Z^0} \sE_a \ar[u] 
 }
\]
defines the image of 
\[
 \left( (S^{\sharp},\sE_{b} \to \sE_{b'}), 
 (S^{\sharp},\sE_{a} \to \sE_{a'}) \right) 
\]
under $t_{b,b',a,a'}^{\mu, \lambda}$ in 
$\Sht_{G,ab,a'b',\Spd \breve{E}}^{\mu-\lambda}$. 
Note that we have equalities 
$G_b(F)=G_{ab}(F)$ and $G_{b'}(F)=G_{a'b'}(F)$. 

\begin{prop}\label{prop:twist}
We have 
\[
 \left( R\Gamma_{\mathrm{c}} (\Sht_{G,b,b'}^{\mu}) \otimes^{\bL} 
 R\Gamma_{\mathrm{c}} (\Sht_{Z^0,a,a'}^{\lambda}) 
 \right) \otimes_{Z^0(F)}^{\bL} \ol{\bQ}_{\ell} \simeq 
 R\Gamma_{\mathrm{c}} (\Sht_{G,ab,a'b'}^{\mu-\lambda}) 
\]
in the derived category of representations of 
$G_b(F) \times G_{b'}(F) \times W_E$. 
\end{prop}
\begin{proof}
This follows from Lemma \ref{lem:torhom} and that 
$t_{b,b',a,a'}^{\mu, \lambda}$ is 
a $Z^0(F)$-torsor. 
\end{proof}

\section{Formula on cohomology}\label{sec:coh}
Let $b_0 ,\ldots , b_m \in G(\breve{F})$ and 
$\mu_1 , \ldots , \mu_m \in X_* (T)^+$. 
Let $E$ be a finite extension of $F$ in $\bC_p$ 
containing $E_i$ for $1 \leq i \leq m$. 
%Assume that $[b_{i-1}] \in B(G,\mu_i,[b_i])$ for $1 \leq i \leq m$. 
Let 
\[
 m_{b_{\bullet},\mu_{\bullet}} \colon 
 \Sht_{b_0,b_1,\Spd \breve{E}}^{\mu_1} \times_{\Spd \breve{E}} \cdots 
 \times_{\Spd \breve{E}} \Sht_{b_{m-1},b_m,\Spd \breve{E}}^{\mu_m} \to 
 \Sht_{b_0,b_m,\Spd \breve{E}}^{\lvert \mu_{\bullet} \rvert} 
\]
by the pullback of 
the convolution morphism 
$m_{b_{\bullet},\mu_{\bullet},U_m}$ defined in \S \ref{sec:conv} 
under the morphism 
\[
 \Spd \breve{E}=\Delta_{m,\Spd \breve{E}} 
 \hookrightarrow (\Spd \breve{E})^m \lra 
 \Spd \breve{E_1} \times \cdots \times \Spd \breve{E_m}. 
\] 
The morphism $m_{b_{\bullet},\mu_{\bullet}}$ 
coincides with 
the morphism defined by 
the composition of modifications. 
This induces 
\[
 \overline{m}_{b_{\bullet},\mu_{\bullet}} \colon 
 (\Sht_{b_0,b_1,\Spd \breve{E}}^{\mu_1} \times_{\Spd \breve{E}} \cdots 
 \times_{\Spd \breve{E}} \Sht_{b_{m-1},b_m,\Spd \breve{E}}^{\mu_m})/
 (\wt{J}_{b_1} \times \cdots \times \wt{J}_{b_{m-1}}) \to 
 \Sht_{b_0,b_m,\Spd \breve{E}}^{\lvert \mu_{\bullet} \rvert}, 
\]
where $\wt{J}_{b_i}$ for $1 \leq i \leq m-1$ 
acts diagonally on the factor 
\[
 \Sht_{b_{i-1},b_i,\Spd \breve{E}}^{\mu_i} \times_{\Spd \breve{E}} 
 \Sht_{b_i,b_{i+1},\Spd \breve{E}}^{\mu_{i+1}} 
\]
and trivially on the other factors. 

Let 
\[
 \wt{\mathrm{Gr}}_{G,\Spd \breve{E},\leq \mu_{\bullet}} 
 \xrightarrow{m_{\mu_{\bullet}}} 
 \mathrm{Gr}_{G,\Spd \breve{E},\leq \lvert \mu_{\bullet} \rvert} 
\]
be the pullback of 
\[
 m_{\mu_{\bullet}} \colon 
 \wt{\Gr}_{G,\Spd E_1 \times \cdots \times \Spd E_m,\leq \mu_{\bullet}} 
 \lra 
 \Gr_{G,\Spd E_1 \times \cdots \times \Spd E_m,\leq \mu_{\bullet}} 
\]
under 
\[
 \Spd \breve{E} =\Delta_{m,\Spd \breve{E}} \hookrightarrow 
 (\Spd \breve{E})^m \lra 
 \Spd E_1 \times \cdots \times \Spd E_m . 
\]
We define  
$m_{\mu_{\bullet},b_0,b_m} \colon 
 \Sht_{b_0,b_m,\Spd \breve{E}}^{\mu_{\bullet}} \to 
 \Sht_{b_0,b_m,\Spd \breve{E}}^{\lvert \mu_{\bullet} \rvert}$ 
by the fiber product 
\[
 \xymatrix{
 \Sht_{b_0,b_m,\Spd \breve{E}}^{\mu_{\bullet}} 
 \ar[rr]^-{m_{\mu_{\bullet},b_0,b_m}} \ar[d] & &
 \Sht_{b_0,b_m,\Spd \breve{E}}^{\lvert \mu_{\bullet} \rvert} \ar[d] \\ 
 \wt{\mathrm{Gr}}_{G,\Spd \breve{E},\leq \mu_{\bullet}} 
 \ar[rr]^-{m_{\mu_{\bullet}}} & & 
 \mathrm{Gr}_{G,\Spd \breve{E},\leq \lvert \mu_{\bullet} \rvert} . 
 }
\]
Then 
$\Sht_{b_0,b_m}^{\mu_{\bullet}}$ is a moduli space of 
modifications 
\[
 \sE_{b_0} \xrightarrow{f_1} \sE_1 \xrightarrow{f_2} \cdots 
 \xrightarrow{f_{m-1}} 
 \sE_{m-1} \xrightarrow{f_m} \sE_{b_m} 
\]
at $S^{\sharp}$ 
such that $f_i$ is bounded by $\mu_i$ for $1 \leq i \leq m$. 
We define a subspace 
$\Sht_{b_0,b_m,\Spd \breve{E}}^{b_1, \ldots, b_{m-1}, \mu_{\bullet}} \subset 
 \Sht_{b_0,b_m,\Spd \breve{E}}^{\mu_{\bullet}}$ as 
a moduli space of modifications 
\[
 \sE_{b_0} \xrightarrow{f_1} \sE_1 \xrightarrow{f_2} \cdots 
 \xrightarrow{f_{m-1}} 
 \sE_{m-1} \xrightarrow{f_m} \sE_{b_m} 
\]
at $S^{\sharp}$ 
such that 
$f_i$ is bounded by $\mu_i$ for $1 \leq i \leq m$ 
and 
$\sE_i$ is isomorphic to 
$\sE_{b_i}$ geometric fiberwisely 
for $1 \leq i \leq m-1$. 

We put 
\[
 I_{b_0,b_m}^{\mu_{\bullet}} = 
 \{ ([b_1], \ldots, [b_{m-1}]) \in B(G)^{m-1} \mid 
 \Sht_{b_{i-1},b_i}^{\mu_i} \neq \emptyset 
 \ \textrm{for} \ 1 \leq i \leq m \}. 
\]
We take $\mu_{m+1}$ such that 
$[b_m] \in B(G,\mu_{m+1},[1])$. 
Then $I_{b_0,b_m}^{\mu_{\bullet}}$ 
is a finite set, since it is contained in 
$\prod_{1 \leq i \leq m-1} B(G,\sum_{j=i+1}^{m+1} \mu_j ,[1] )$ 
by Lemma \ref{lem:Mempty}. 
For $\lambda \in X_*(T)^+/\Gamma_F$, we put 
\[
 V_{\mu_{\bullet}}^{\lambda} = 
 \Hom_{{}^L G} (V_{\lambda},\bigotimes_{1 \leq i \leq m} V_{\mu_i}) . 
\]
For 
$([b_i] )_{1 \leq i \leq m-1} \in I_{b_0,b_m}^{\mu_{\bullet}}$, 
we put 
$N_{b_{\bullet}}=\sum_{1 \leq i \leq m-1}N_{b_i}$. 
We write $\Gr_{G,\Spd E,\leq \mu}^{(\infty)}$ 
for the inverse image of 
$\Gr_{G,\Spd E,\leq \mu}$ 
under $LG_{\Spd E} \to \Gr_{G,\Spd E}$. 

\begin{prop}\label{prop:decomp}
The sum 
\[
 \sum_{\lambda \in X_*(T)^+/\Gamma} V_{\mu_{\bullet}}^{\lambda} \otimes^{\bL} 
 R\Gamma_{\mathrm{c}} (\Sht_{b_0,b_m}^{\lambda}) 
\]
is decomposed into 
\[
 \left( \bigotimes_{1 \leq i \leq m}  R\Gamma_{\mathrm{c}} 
 (\Sht_{b_{i-1},b_i}^{\mu_i}) \otimes^{\bL} 
 \bigotimes_{1 \leq i \leq m-1} \delta_{b_i}  
 \right) \otimes_{\prod_{i=1}^{m-1} G_{b_i}(F)}^{\bL} \Lambda 
 [2N_{b_{\bullet}}] 
\]
for 
$([b_i] )_{1 \leq i \leq m-1} \in I_{b_0,b_m}^{\mu_{\bullet}}$ 
by distinguished triangles 
in the derived category of representations of 
$G_{b_0}(F) \times G_{b_m}(F) \times W_E$. 
\end{prop}

\begin{proof}
Let $\mathrm{IC}_{\mu_{\bullet}}$ be the 
external twisted product of 
$\IC_{\mu_1}, \ldots ,\IC_{\mu_m}$ 
on 
$\wt{\mathrm{Gr}}_{\Spd \breve{E},\leq \mu_{\bullet}}$. 
By the construction of convolution product \cite[VI.8]{FaScGeomLLC} in geometric Satake equivalence and \cite[Proposition VII.4.3]{FaScGeomLLC}, 
we have 
\[
 (m_{\mu_{\bullet}})_{\natural} \mathrm{IC}_{\mu_{\bullet}} = 
 \sum_{\lambda \in X_*(T)^+/\Gamma} V_{\mu_{\bullet}}^{\lambda} \otimes^{\bL} 
 \mathrm{IC}_{\lambda}. 
\]
Hence the sum 
\[
 \sum_{\lambda \in X_*(T)^+/\Gamma} V_{\mu_{\bullet}}^{\lambda} \otimes^{\bL} 
 R\Gamma_{\mathrm{c}} (\Sht_{b_0,b_m}^{\lambda}) 
\]
is isomorphic to 
$R\Gamma_{\mathrm{c}} (\Sht_{b_0,b_m}^{\mu_{\bullet}}, 
 \mathrm{IC}_{\mu_{\bullet}})$. 
 
We put $\mu'_{\bullet}=(\mu_1,\ldots ,\mu_{m-2})$. 
Let $\{ [b_{m-1}^j] \}_{1 \leq j \leq n}$ be the image of the projection 
$I_{b_0,b_m}^{\mu_{\bullet}} \to B(G)$ to the $(m-1)$-th component. 
It suffices to show that 
$R\Gamma_{\mathrm{c}} (\Sht_{b_0,b_m}^{\mu_{\bullet}}, 
\mathrm{IC}_{\mu_{\bullet}})$ is decomposed into 
\begin{equation}\label{eq:decm-1}
\left(R\Gamma_{\mathrm{c}} 
(\Sht_{b_{0},b_{m-1}^j}^{\mu'_{\bullet}}) \otimes^{\bL} 
R\Gamma_{\mathrm{c}} 
(\Sht_{b_{m-1}^j,b_m}^{\mu_m}) 
\otimes^{\bL} \delta_{b_{m-1}^j} 
\right) \otimes_{G_{b_{m-1}^j}(F)}^{\bL} \Lambda 
[2N_{b_{m-1}^j}] 
\end{equation}
for $1 \leq j \leq n$. 

Let $K \subset G_{b_0}(F)$ be enough small compact open subgroup. 
Then 
\[
 R\Gamma_{\mathrm{c}} (\Sht_{b_0,K,b_m}^{\mu_{\bullet}}, 
 \mathrm{IC}_{\mu_{\bullet}}) \cong 
 t_{b_m}^* i_{b_m}^* T_{\mu_{\bullet}} i_{b_0,!} (h_{K,!} \Lambda)
 \cong 
 t_{b_m}^* i_{b_m}^* T_{\mu_{m-1}}
 T_{\mu'_{\bullet}} i_{b_0,!} (h_{K,!} \Lambda)
\] 
is decomposed into 
\begin{equation*}
  t_{b_m}^* i_{b_m}^* T_{\mu_{m-1}} 
 i_{b_{m-1}^j,!} i_{b_{m-1}^j}^*
 T_{\mu'_{\bullet}} i_{b_0,!} (h_{K,!} \Lambda)
\end{equation*}
for $1 \leq j \leq n$ by 
Lemma \ref{lem:jjss} \ref{en:jjiiBun}. 
This is isomorphic to 
\begin{equation}\label{eq:Tm-1mu'}
  t_{b_m}^* i_{b_m}^* T_{\mu_{m-1}} 
i_{b_{m-1}^j,!} \bigl( \delta_{b_{m-1}^j} [2N_{b_{m-1}^j}] \otimes^{\bL} h_{b_{m-1}^j,!} h_{b_{m-1}^j}^* i_{b_{m-1}^j}^*
T_{\mu'_{\bullet}} i_{b_0,!} (h_{K,!} \Lambda) \bigr)
\end{equation}
by Lemma \ref{lem:Jbnat}. 
By Lemma \ref{lem:torhom}, 
\[
 h_{b_{m-1}^j}^* i_{b_{m-1}^j}^*
 T_{\mu'_{\bullet}} i_{b_0,!} (h_{K,!} \Lambda) \cong 
 \left( \Bigl( \varinjlim_{K' \subset G_{b_{m-1}^j}(F)} h_{K',\natural} \Lambda \Bigr) \otimes^{\bL} R\Gamma_{\mathrm{c}} 
 (\Sht_{b_{0},K,b_{m-1}^j}^{\mu'_{\bullet}}) \right) \otimes_{\cH(G_{b_{m-1}^j}(F))}^{\bL} \Lambda . 
\]
Hence \eqref{eq:Tm-1mu'} is isomorphic to 
\begin{align*}
%&\left( t_{b_m}^* i_{b_m}^* T_{\mu_{m-1}} i_{b_{m-1}^j,!} \Bigl( \varinjlim_{K' \subset G_{b_{m-1}^j}(F)} h_{K',\natural} \Lambda \Bigr) \otimes^{\bL} R\Gamma_{\mathrm{c}} (\Sht_{b_{0},K,b_{m-1}^j}^{\mu'_{\bullet}}) \otimes^{\bL} \delta_{b_{m-1}^j} [2N_{b_{m-1}^j}] \right) \otimes_{\cH(G_{b_{m-1}^j}(F))} \Lambda \\ &\cong 
\left(R\Gamma_{\mathrm{c}} 
(\Sht_{b_{0},K,b_{m-1}^j}^{\mu'_{\bullet}}) \otimes^{\bL} 
R\Gamma_{\mathrm{c}} 
(\Sht_{b_{m-1}^j,b_m}^{\mu_m}) 
\otimes^{\bL} \delta_{b_{m-1}^j} 
\right) \otimes_{G_{b_{m-1}^j}(F)}^{\bL} \Lambda 
[2N_{b_{m-1}^j}] 
\end{align*}
since $t_{b_m}^* i_{b_m}^* T_{\mu_{m-1}} 
i_{b_{m-1}^j,!}$ commutes with direct limits, tensors and changes of coefficients. 
Therefore we obtain the claim. 
\end{proof}

\begin{cor}\label{cor:formula}
We have 
\begin{align*}
 \sum_{([b_i]
 )_{1 \leq i \leq m-1} \in I_{b_0,b_m}^{\mu_{\bullet}}} 
 H_* 
 \left( \prod_{i=1}^{m-1} G_{b_i}(F), 
 \bigotimes_{1 \leq i \leq m}  H_{\mathrm{c}}^* 
 (\Sht_{b_{i-1},b_i}^{\mu_i}) \otimes^{\bL} 
 \bigotimes_{1 \leq i \leq m-1} \delta_{b_i} 
 \right) & \\ 
 =
 \sum_{\lambda \in X_*(T)^+/\Gamma} V_{\mu_{\bullet}}^{\lambda} \otimes^{\bL} & 
 H_{\mathrm{c}}^* 
 (\Sht_{b_0,b_m}^{\lambda}) 
\end{align*}
as virtual representations of 
$G_{b_0}(F) \times G_{b_m}(F) \times W_E$. 
\end{cor}
\begin{proof}
This follows from 
Proposition \ref{prop:decomp} by taking cohomology. 
\end{proof}

\begin{lem}\label{lem:RHom}
Assume that $m=2$. 
Let $\pi$ be a smooth representation of $G_{b_0} (F)$. 
Then we have 
\begin{align*}
 R\Hom_{G_{b_0}(F)} & \left( (R\Gamma_{\mathrm{c}} (\Sht_{b_0,b_1}^{\mu_1}) 
 \otimes^{\bL} 
 R\Gamma_{\mathrm{c}} (\Sht_{b_1,b_2}^{\mu_2}) \otimes^{\bL} \delta_{b_1}) \otimes_{G_{b_1}(F)}^{\bL} 
 \Lambda 
 ,\pi \right) \\ 
 &\simeq R\Hom_{G_{b_1} (F)} \left( 
 R\Gamma_{\mathrm{c}} (\Sht_{b_1,b_2}^{\mu_2}) ,
 R\Hom_{G_{b_0}(F)} 
 \left( R\Gamma_{\mathrm{c}} (\Sht_{b_0,b_1}^{\mu_1}),\pi \right) \otimes^{\bL} \delta_{b_1}^{-1} \right) 
\end{align*}
in the derived category of representations of 
$G_{b_2}(F) \times W_E$ 
for $[b_1] \in I_{b_0,b_2}^{(\mu_1,\mu_2)}$. 
\end{lem}
\begin{proof}
We have 
\begin{align*}
 R\Hom_{G_{b_0}(F)} & \left( R\Gamma_{\mathrm{c}} (\Sht_{b_0,b_1}^{\mu_1}) 
 \otimes
 R\Gamma_{\mathrm{c}} (\Sht_{b_1,b_2}^{\mu_2}) \otimes^{\bL} \delta_{b_1} \otimes_{G_{b_1}(F)}^{\bL} 
 \Lambda ,\pi \right) \\ 
 &\simeq R\Hom_{G_{b_0}(F) \times G_{b_1} (F)} 
 \left( R\Gamma_{\mathrm{c}} (\Sht_{b_0,b_1}^{\mu_1}) \otimes
 R\Gamma_{\mathrm{c}} (\Sht_{b_1,b_2}^{\mu_2}) \otimes^{\bL} \delta_{b_1}, 
 \Lambda \boxtimes \pi 
 \right) \\ 
 &\simeq R\Hom_{G_{b_0}(F) \times G_{b_1} (F)} \left( 
 R\Gamma_{\mathrm{c}} (\Sht_{b_1,b_2}^{\mu_2}) \otimes^{\bL} \delta_{b_1}, 
 \Hom \left( R\Gamma_{\mathrm{c}} (\Sht_{b_0,b_1}^{\mu_1}),\pi \right) \right) \\ 
 &\simeq R\Hom_{G_{b_1} (F)} \left( 
 R\Gamma_{\mathrm{c}} (\Sht_{b_1,b_2}^{\mu_2}) ,
 R\Hom_{G_{b_0}(F)} 
 \left( R\Gamma_{\mathrm{c}} (\Sht_{b_0,b_1}^{\mu_1}),\pi \right) \otimes^{\bL} \delta_{b_1}^{-1} \right) 
\end{align*}
in the derived category of representations of 
$G_{b_2}(F) \times W_E$. 
\end{proof}

\section{Duality morphism}\label{sec:Dual}
Assume that $2$ is invertible in $\Lambda$. 
We take a pinning $\cP=(G,B,T,{X_{\alpha}})$ of $G$. 
Then define a duality involution 
$\iota_{G,\cP}$ on $G$ as in 
\cite[Definition 1]{PraMVWinv}. 
We simply write $\iota$ for $\iota_{G,\cP}$. 
Note that $\mu=-\iota \circ \mu$ in 
$X_*(T)/W_G(T) \cong X_*(T)^+$. 
We define an anti-involution 
$\theta$ on $G$ 
by $\theta (g)=\iota(g)^{-1}$.  
We define the duality morphism 
\[
 \theta_{b,b'} \colon 
 \Sht_{G,b,b'}^{\mu} \lra 
 \Sht_{G,\iota(b'),\iota(b)}^{\mu} 
\]
by sending $f \colon \sE_b \to \sE_{b'}$ to 
$\iota(f)^{-1} \colon \sE_{\iota(b')} \to \sE_{\iota(b)}$. 
The above isomorphism is compatible with actions of 
$\wt{J}_b \times \wt{J}_{b'}$ and 
$\wt{J}_{\iota(b')} \times \wt{J}_{\iota(b)}$ 
under the isomorphism 
\[
 \wt{J}_b \times \wt{J}_{b'} \lra 
 \wt{J}_{\iota(b')} \times \wt{J}_{\iota(b)} ;\ 
 (g,g') \mapsto (\iota(g'),\iota(g)). 
\]
Then 
$\theta_{b,\iota(b)}$ is an involution on 
$\Sht_{G,b,\iota(b)}^{\mu}$. 
%We put $\Gr_G^{\mathrm{op}}=L^+G \backslash LG$. 
On the other hand, $\theta$ induces a morphism 
$\theta \colon \cHck_G \to \cHck_G$. 
Let $E$ be the field of definition of $\mu$. 
We have a natural morphism 
\[ 
p_{b,b'}^{\mu} \colon  \Sht_{G,b,b'}^{\mu} \lra \cHck_{G,\Spd \breve{E}}.  
\]
%\[ \pi_{b,b'}^{\mu,\mathrm{op}} \colon  \Sht_{G,b,b'}^{\mu} \lra \Gr_{G,\Spd \breve{E}}^{\mathrm{op}} \]
%obtained by forgetting the trivialization of $\sE_b$. 
We have the commutative diagram 
\[
\xymatrix{
	\Sht_{G,b,b'}^{\mu} \ar[d]_-{p_{b,b'}^{\mu}} 
	\ar[r]^-{\theta_{b,b'}} & 
	\Sht_{G,\iota(b'),\iota(b)}^{\mu} 
	\ar[d]^-{p_{\iota(b'),\iota(b)}^{\mu}} \\ 
	\cHck_{G,\Spd \breve{E}} 
	\ar[r]^-{\theta} & 
	\cHck_{G,\Spd \breve{E}} . 
}
\]
%\[ \xymatrix{ \Sht_{G,b,b'}^{\mu} \ar[d]_-{\pi_{b,b'}^{\mu,\mathrm{op}}}  \ar[r]^-{\theta_{b,b'}} &  \Sht_{G,\iota(b'),\iota(b)}^{\mu}  \ar[d]^-{\pi_{\iota(b'),\iota(b)}^{\mu}} \\  \Gr_{G,\Spd \breve{E}}^{\mathrm{op}}  \ar[r]^-{\theta} &  \Gr_{G,\Spd \breve{E}} .  }\]
We have $\cS'(r_{\mu} \circ \ad (\widehat{\rho}(-1)))\cong \theta^* \IC_{\mu}'$ by \cite[Proposition VI.12.1]{FaScGeomLLC}. 
Hence $\widehat{\rho}(-1) \colon r_{\mu} \circ \ad (\widehat{\rho}(-1)) \to r_{\mu}$ induces 
$M_{\mu} \colon \theta^* \IC_{\mu}' \to \IC_{\mu}'$. 
%as in \cite[(2.4.2)]{ZhuAffGmix}. 
%Further, we have a canonical isomorphism $(\pi_{b,b'}^{\mu,\mathrm{op}})^* \IC_{\mu}^{\mathrm{op}} \to  (\pi_{b,b'}^{\mu})^* \IC_{\mu}$ as in \cite[Lemma 2.24]{ZhuAffGmix}. 
Hence we obtain the isomorphism 
\[
 R\Gamma_{\mathrm{c}} (\Sht_{G,\iota(b'),\iota(b)}^{\mu}) 
 \to R\Gamma_{\mathrm{c}} (\Sht_{G,b,b'}^{\mu}) 
\]
induced by $\theta_{b,b'}$. 

\begin{lem}\label{lem:dualiso}
The isomorphism 
\[
 R\Gamma_{\mathrm{c}} (\Sht_{G,\iota(b'),\iota(b)}^{\mu}) 
 \to 
 R\Gamma_{\mathrm{c}} (\Sht_{G,b,b'}^{\mu}) 
\]
is compatible with actions of 
$\wt{J}_b \times \wt{J}_{b'}$ and 
$\wt{J}_{\iota(b')} \times \wt{J}_{\iota(b)}$ 
under the isomorphism 
\[
 \wt{J}_b \times \wt{J}_{b'} \lra 
 \wt{J}_{\iota(b')} \times \wt{J}_{\iota(b)} ;\ 
 (g,g') \mapsto (\iota(g'),\iota(g)). 
\]
\end{lem}
\begin{proof}
This follows from the definition. 
\end{proof}

Further, we have an involution 
\[
 \theta_b \colon 
 \Sht_{G,b,1}^{\mu} \times \Sht_{G,1,\iota(b)}^{\mu} \lra 
 \Sht_{G,b,1}^{\mu} \times \Sht_{G,1,\iota(b)}^{\mu};\ 
 (x,x') \mapsto (\theta_{1,\iota(b)}(x'),\theta_{b,1}(x)). 
\]
We have a decomposition 
\[
 V_{\mu} \otimes V_{\mu} =\Sym^2 V_{\mu} \oplus \bigwedge^2 V_{\mu}. 
\]
Let 
\[
 \Psi_{b,\mu} \colon \left(
 R\Gamma_{\mathrm{c}} (\Sht_{b,1}^{\mu}) \otimes 
 R\Gamma_{\mathrm{c}} (\Sht_{1,\iota(b)}^{\mu}) \right) 
 \otimes_{G(F)}^{\bL} \Lambda
 \to 
 \sum_{\lambda \in X_*(T)^+/\Gamma} V_{\mu_{\bullet}}^{\lambda} \otimes  
 R\Gamma_{\mathrm{c}} 
 (\Sht_{b,\iota(b)}^{\lambda}) 
\]
be the morphism given by 
Proposition \ref{prop:decomp}. 
Let $s_{b,\mu}$ be the involution on 
the source of $\Psi_{b,\mu}$ induced by $\theta_b$ 
and the multiplication by $(-1)^{\langle 2\wh{\rho},\mu \rangle}$. 
On the other hand, 
let $t_{b,\mu}$ be the involution on 
the target of $\Psi_{b,\mu}$ induced by 
the permutation 
$\sigma_{V_{\mu},V_{\mu}}$ on $V_{\mu} \otimes V_{\mu}$ 
and 
$\theta_{b,\iota(b)} \colon \Sht_{b,\iota(b)}^{\lambda} \to \Sht_{b,\iota(b)}^{\lambda}$.

\begin{prop}\label{prop:invcomp}
The morphism $\Psi_{b,\mu}$ is compatible 
with the involutions 
$s_{b,\mu}$ and $t_{b,\mu}$. 
\end{prop}
\begin{proof}
By the characterization of 
the commutativity constraint, 
the equality 
\[
 \IC_{\mu}' \star \IC_{\mu}' =
 \sum_{\lambda \in X_*(T)^+/\Gamma} V_{\mu_{\bullet}}^{\lambda} \otimes 
 \mathrm{IC}_{\lambda}' 
\]
is compatible with the involutions 
$c_{V_{\mu}, V_{\mu}}$ and $\sigma_{V_{\mu}, V_{\mu}}$. 
Hence the target of 
$\Psi_{b,\mu}$ is equal to 
$H_{\mathrm{c}}^* (\Sht_{b,\iota(b)}^{2\mu},\IC_{\mu}' \star \IC_{\mu}')$ 
with the involution given by $c_{V_{\mu}, V_{\mu}}$ and $\theta_{b,\iota(b)}$. 
Let $\sigma_{2,X} \colon (\mathrm{Div}_X^1)^2 \to (\mathrm{Div}_X^1)^2$ 
and 
$\sigma_{2,G} \colon \cHck_G^{\{ 1,2 \}} \to \cHck_G^{\{ 1,2 \}}$ 
be the permutation of two Cartier divisors. 
Let $\IC_{\mu}' * \IC_{\mu}'$ be the fusion product on 
$\cHck_G^{\{ 1,2 \}}$. 
Here we use the notation at the beginning of 
\cite[VI.9]{FaScGeomLLC}. 
Then we have the morphism 
\[
\widetilde{c}_{V_{\mu},V_{\mu}} \colon \sigma_{2,G}^*(\IC_{\mu}' * \IC_{\mu}') \to \IC_{\mu}' * \IC_{\mu}' 
\] 
extending 
$c_{V_{\mu}, V_{\mu}}$. 

The morphism $\theta$ induces 
$\theta^{\{ 1\},\{2 \}} \colon \cHck_G^{\{ 1,2 \};\{ 1\},\{2 \}} \to \cHck_G^{\{ 1,2 \};\{ 1\},\{2 \}}$ 
switching two Cartier divisors. 
Here we use the notation in the proof of 
\cite[Proposition VI.9.4]{FaScGeomLLC}. 
Then we have a morphism 
\[
 S_{\mu,\mu} \colon \theta^{\{ 1\},\{2 \}*} (\IC_{\mu}' \boxtimes \IC_{\mu}') \to \IC_{\mu}' \boxtimes \IC_{\mu}' 
\]
induced by $M_{\mu}$ and switching two factors of $\IC_{\mu}'$. 
The morphism $\theta$ induces 
$\theta^{\{ 1,2 \}} \colon \cHck_G^{\{ 1,2 \}} \to \cHck_G^{\{ 1,2 \}}$ 
switching two Cartier divisors. 
Then we have 
\[
S_{\mu,\mu}' = m_{\natural} (S_{\mu,\mu}) \colon \theta^{\{1,2\}*}(\IC_{\mu}' * \IC_{\mu}') \to \IC_{\mu}' * \IC_{\mu}'. 
\]
Since $\theta^{\{1,2\}} \circ \sigma_{2,G}$ is the automorphism of 
$\cHck_G^{\{ 1,2 \}}$ over $(\mathrm{Div}_X^1)^2$ induced by $\theta$, 
we also have 
$M_{\mu,\mu} \colon (\theta^{\{1,2\}} \circ \sigma_{2,G})^*(\IC_{\mu}' * \IC_{\mu}') \to \IC_{\mu}' * \IC_{\mu}'$ 
defined in the same way as $M_{\mu}$. 

%We have $\cS'(r_{\mu,\mu} \circ \ad (\widehat{\rho}(-1)) \circ \mathrm{sw}_G)\cong \sigma_{2,G}^* \theta^{\{ 1,2 \}*} (\IC_{\mu} * \IC_{\mu})$ by \cite[Proposition VI.12.1]{FaScGeomLLC}. 
%Hence $\widehat{\rho}(-1) \colon r_{\mu,\mu} \circ \ad (\widehat{\rho}(-1)) \to r_{\mu,\mu}$ induces $M_{\mu,\mu} \colon \theta^{\{ 1,2 \}*} (\IC_{\mu} * \IC_{\mu}) \to \IC_{\mu} * \IC_{\mu}$. 
Then it suffices to show that 
\[
 \theta^{\{1,2\},*}(\IC_{\mu}' * \IC_{\mu}')
 \xrightarrow{\sigma_{2,G}^*(M_{\mu,\mu})}  
 \sigma_{2,G}^*(\IC_{\mu}' * \IC_{\mu}')) \xrightarrow{\widetilde{c}_{V_{\mu},V_{\mu}}}  
 \IC_{\mu}' * \IC_{\mu}'
\]
and $S_{\mu,\mu}'$ 
are equal. 
It suffices to check 
this on $\cHck_G^{\{1,2\}} \times_{(\mathrm{Div}_X^1)^{\{1,2\}}} (\mathrm{Div}_X^1)^{\{1,2\};\{1\},\{2\}}$ 
by \cite[Proposition VI.9.3]{FaScGeomLLC}. 
This follows from the constructions of 
$\widetilde{c}_{V_{\mu},V_{\mu}}$, $S'_{\mu,\mu}$ and $M_{\mu,\mu}$. 
\end{proof}

\section{Kottwitz conjecture}\label{sec:Kot}

\begin{defn}
Let $\varphi \colon W_F \to {}^LG$ 
be an $\ell$-adic local L-parameter for $G$ (\cf \cite[Definition 1.14]{ImaLLCell}). 
We put 
\[
 S_{\varphi}=
 \{ g \in \wh{G}(\ol{\bQ}_{\ell}) \mid g \varphi g^{-1}=\varphi \}. 
\]
We say that $\varphi$ is discrete if 
$S_{\varphi}/Z(\wh{G})^{\Gamma_F}$ is finite (\cf \cite[Definition 4.1]{FarGover}). 
\end{defn}

Let $b,b' \in \GL_n(\breve{F})$ 
such that 
$[b] \in B(G,\mu,[b'])$. 
We put 
\[
 H_{\mathrm{c}}^{\bullet} (\Sht_{b,b'}^{\mu})[\pi] = 
 \sum_{i,j \in \bZ} (-1)^{i+j} \Ext_{G_{b}(F)}^i 
 \left( R^j\Gamma_{\mathrm{c}} (\Sht_{b,b'}^{\mu}), \pi 
 \right) 
\]
for an irreducible smooth representation $\pi$ of $G_{b}(F)$. 

The following is a version of Kottwitz conjecture 
for moduli spaces of mixed characteristic local shtukas 
in $\mathrm{GL}_n$-case (\cf \cite[Conjecture 7.4]{RVlocSh}): 

\begin{conj}\label{conj:Kot}
Assume that $b,b'$ are basic. 
Let $\varphi \colon W_F \to {}^L\GL_n$ 
be a discrete local L-parameter. 
Let $\pi_b$ and $\pi_{b'}$ be the 
irreducible smooth representations of 
$G_b(F)$ and $G_{b'}(F)$ corresponding to 
$\varphi$ via the local Langlands correspondence. 
Then we have 
\[
 H_{\mathrm{c}}^{\bullet} (\Sht_{b,b'}^{\mu})[\pi_{b}]=
 \pi_{b'} \boxtimes (r_{\mu} \circ \varphi) 
\]
in $\mathrm{Groth}(G_{b'}(F) \times W_F)$. 
\end{conj}

For an object $\cC$ in a derived category, we put 
$\cH^*(\cC)=\bigoplus_{i \in \bZ} \cH^i(\cC)$. 
The following conjecture is motivated by 
\cite[Th\'eor\`eme A]{DatLTel}. 

\begin{conj}\label{conj:DKot}
Assume that $b,b'$ are basic. 
Let $\varphi \colon W_F \to {}^L\GL_n$ 
be a discrete local L-parameter. 
Let $\pi_b$ and $\pi_{b'}$ be the irreducible smooth representations of 
$G_b(F)$ and $G_{b'}(F)$ corresponding to 
$\varphi$ via the local Langlands correspondence. 
Then we have 
\[
 \cH^* \left( R\Hom_{G_{b}(F)} 
 \left( R\Gamma_{\mathrm{c}} (\Sht_{b,b'}^{\mu}), \pi_{b} 
 \right) \right) \simeq 
 \pi_{b'} \boxtimes (r_{\mu} \circ \varphi) 
\]
as representations of $G_{b'}(F) \times W_F$. 
\end{conj}

%\begin{rem}
%We have 
%\[ H_{\mathrm{c}}^{\bullet} (\Sht_{b,b'}^{\mu})[\pi] =  \sum_{i \in \bZ} (-1)^i R^i \Hom_{G_{b}(F)}  \left( R\Gamma_{\mathrm{c}} (\Sht_{b,b'}^{\mu}), \pi  \right). \]Hence Conjecture \ref{conj:DKot} is stronger than Conjecture \ref{conj:Kot}. \end{rem}

\begin{lem}\label{lem:contr}
Assume that $b$ is basic. 
Let $\pi_b$ and $\pi_{\iota(b)}$ be the irreducible smooth 
representations of 
$G_b(F)$ and $G_{\iota(b)}(F)$ corresponding via the local Jacquet--Langlands 
correspondence. 
Then the pullback of $\pi_{\iota(b)}$ under the 
isomorphism $\iota \colon G_b(F) \to G_{\iota(b)}(F)$ 
is isomorphic to $\pi_b^*$. 
\end{lem}
\begin{proof}
By \cite[Corollary 1]{PraMVWinv}, 
we may assume that 
$\iota(g)={}^t g^{-1}$. 
If $b=1$, the calim follows from a theorem of 
Gelfand and Kazhdan (\cf \cite[7.3. Theorem]{BeZeRepGL}). 
If regular elements 
$g \in \GL_n(F)$ and $g' \in G_b(F)$ 
have the same reduced characteristic polynomial, 
then 
$\iota(g) \in \GL_n(F)$ and $\iota(g') \in G_{\iota(b)}(F)$ 
are regular and 
have the same reduced characteristic polynomial. 
Hence the claim follows from the case where $b=1$ 
and the characterization of the local Jacquet--Langlands 
correspondence. 
\end{proof}

We put $\kappa (b)=v_F (\det (b))$. 
For $m_1, \ldots , m_n \in \bZ$, 
let 
$(m_1, \ldots , m_n)$ denote the cocharacter 
of $\GL_n$ or its standard Levi subgroup defined by 
$z \mapsto \mathrm{diag} (z^{m_1}, \ldots , z^{m_n})$. 

\begin{thm}\label{thm:Kcases}
Conjecture \ref{conj:DKot} is true 
in the following cases: 
\begin{enumerate}
\item\label{en:d0}
$\kappa (b) \equiv \kappa (b') \mod n$ and 
\[
 \mu =\frac{\kappa (b) -\kappa (b')}{n}(1,\ldots,1). 
\]
\item\label{en:d1}
$\kappa (b) \equiv 0,1,\ \kappa (b) \equiv \kappa (b') + 1 \mod n$ and 
\[
 \mu =\frac{\kappa (b) -\kappa (b')-1}{n}(1,\ldots,1)+
 (1,0,\ldots,0). 
\]
\item\label{en:d-1}
$\kappa (b) \equiv 0,-1,\ \kappa (b) \equiv \kappa (b') - 1 \mod n$ and 
\[
 \mu =\frac{\kappa (b) -\kappa (b')+1}{n}(1,\ldots,1)+
 (0,\ldots,0,-1). 
\]
\item\label{en:d2}
$\kappa (b) \equiv 1,\ \kappa (b') \equiv -1 \mod n$ and 
\[
 \mu =\frac{\kappa (b) -\kappa (b')-2}{n}(1,\ldots,1)+
 \begin{cases}
 (2,0,\ldots,0), \\ 
 (1,1,0,\ldots,0). 
 \end{cases}
\]
\item\label{en:d-2}
$\kappa (b) \equiv -1,\ \kappa (b') \equiv 1 \mod n$ and 
\[
 \mu =\frac{\kappa (b) -\kappa (b')+2}{n}(1,\ldots,1)+
 \begin{cases}
 (0,\ldots,0,-2), \\ 
 (0,\ldots,0,-1,-1). 
 \end{cases}
\]
\end{enumerate}
\end{thm}
\begin{proof}
By the inversing isomorphism 
\eqref{eq:invmor}, 
the claims \ref{en:d-1} and \ref{en:d-2} 
are reduced to 
the claims \ref{en:d1} and \ref{en:d2}. 
By Proposition \ref{prop:twist}, 
we may assume that 
$\kappa (b) = \kappa (b')=0$ in \ref{en:d0}, 
$\kappa (b) = 0,-1$, $\kappa (b) = \kappa (b') + 1$ in \ref{en:d1} 
and 
$\kappa (b) = -1$, $\kappa (b') = 1$ in \ref{en:d2}. 
Further, we may assume that 
$\kappa (b) = 0$ in \ref{en:d1} 
by Lemma \ref{lem:dualiso} and Lemma \ref{lem:contr}. 
Then the claim \ref{en:d0} is trivial. 
The claim \ref{en:d1} follows from 
the proof of \cite[Tho\'er\`eme A]{DatLTel} 
taking care the degree in \cite[Tho\'er\`eme 4.1.2]{DatLTel}. 

We show the claim \ref{en:d2}. 
We may assume that $b'=\iota(b)$. 
We put 
\[
 \mu_1 =(1,0,\ldots,0), \ 
 \mu_2=(2,0,\ldots,0), \ 
 \mu_{1,1}=(1,1,0,\ldots,0). 
\]
Note that we have $I_{b,\iota(b)}^{(\mu_1,\mu_1)} =\{ [1] \}$. 
Let $\pi_1$ be 
the irreducible smooth representations of 
$\GL_n(F)$ corresponding to 
$\varphi$ via the local Langlands correspondence. 
By Proposition \ref{prop:decomp}, Lemma \ref{lem:RHom}, 
the claim \ref{en:d1} and \cite[Corollaire 4.2.1]{DatLTel}, 
we have 
\begin{align*}
 (V_{(\mu_1,\mu_1)}^{\mu_2}&)^* \otimes 
 \cH^*  \left( R\Hom_{G_{\iota(b)}(F)}  
 \left( R\Gamma_{\mathrm{c}} (\Sht_{\iota(b),b}^{\mu_2}),\pi_{\iota(b)} \right) \right)  \\ 
 +(&V_{(\mu_1,\mu_1)}^{\mu_{1,1}})^* \otimes 
 \cH^*  \left( R\Hom_{G_{\iota(b)}(F)} 
 \left( R\Gamma_{\mathrm{c}} (\Sht_{\iota(b),b}^{\mu_{1,1}}),\pi_{\iota(b)} \right) \right) \\ 
 &\simeq 
 \cH^* \left( R\Hom_{G_{\iota(b)}(F)} \left( R\Gamma_{\mathrm{c}} (\Sht_{\iota(b),1}^{\mu_1}) \otimes 
 R\Gamma_{\mathrm{c}} (\Sht_{1,b}^{\mu_1}) \otimes_{\GL_n(F)}^{\bL} 
 \ol{\bQ}_{\ell} ,\pi_{\iota(b)} \right) \right) \\ 
 &\simeq 
 \cH^* \left( R\Hom_{\GL_n (F)} \left( 
 R\Gamma_{\mathrm{c}} (\Sht_{1,b}^{\mu_1}) ,
 R\Hom_{G_{\iota(b)}(F)} 
 \left( R\Gamma_{\mathrm{c}} (\Sht_{{\iota(b)},1}^{\mu_1}),\pi_{\iota(b)} \right) \right) \right) \\ 
 &\simeq 
\cH^* \left( R\Hom_{\GL_n (F)} \left( 
R\Gamma_{\mathrm{c}} (\Sht_{1,b}^{\mu_1}) ,
\cH^* \left( R\Hom_{G_{\iota(b)}(F)} 
\left( R\Gamma_{\mathrm{c}} (\Sht_{{\iota(b)},1}^{\mu_1}),\pi_{\iota(b)} \right) \right) \right) \right) \\ 
 &\simeq 
 \cH^* \left( R\Hom_{\GL_n (F)} \left( 
 R\Gamma_{\mathrm{c}} (\Sht_{1,b}^{\mu_1}) ,
 \pi_1 \boxtimes \varphi 
 \right) \right) \\ 
 &\simeq \pi_{b} \boxtimes (\varphi \otimes \varphi) 
 \simeq \pi_{b} \boxtimes \left( 
 (r_{\mu_2} \circ \varphi) \oplus (r_{\mu_{1,1}} \circ \varphi) 
 \right). 
\end{align*}
Using Proposition \ref{prop:invcomp}, 
we can separate the above equality to obtain the claim. 
\end{proof}

\begin{cor}\label{cor:minu}
Conjecture \ref{conj:DKot} is true if $n \leq 3$ and 
$\mu$ is minuscule. 
\end{cor}
\begin{proof}
All the cases are contained in Theorem \ref{thm:Kcases}. 
\end{proof}

\section{Inductive formula}\label{sec:ind}
For a smooth representation $\pi$ of $G(F)$ 
and the unipotent radical $N$ of a parabolic 
subgroup of $G$, 
let $\pi_N$ denote the Jacquet module of 
$\pi$ with respect to $N$. 

Assume that $G=\GL_2$. 
Let $T$ be the diagonal torus and 
$B$ be the upper triangle Borel subgroup of 
$\GL_2$. 
Let $N$ be the unipotent radical of $B$, 
and $N^{\mathrm{op}}$ be the 
the unipotent radical of the opposite 
Borel subgroup $B^{\mathrm{op}}$. 
Let $\delta_B \colon T(F) \to \ol{\bQ}_{\ell}^{\times}$ be the 
modulus character with respect to $B$. 
For 
$b=\begin{pmatrix}
 \varpi^{m} & 0 \\ 
 0 & \varpi^l 
 \end{pmatrix}$ 
with $m <l$, let $\delta_b \colon G_b(F) \to \ol{\bQ}_{\ell}^{\times}$ 
be the character determined by 
$\delta_B$ and the natural isomorphism $G_b(F) \cong T(F)$. 

\begin{lem}\label{lem:(1,0)pi}
Let $m \in \bZ$. 
We put 
\[
 b=\begin{pmatrix}
 \varpi^{m} & 0 \\ 
 0 & \varpi^m 
 \end{pmatrix}, \quad 
 b'=\begin{pmatrix}
 \varpi^{m-1} & 0 \\ 
 0 & \varpi^m 
 \end{pmatrix} . 
\] 
Let $\pi$ be an admissible representation of $G(F)$. 
Then we have 
\[
  R^{\bullet}\Hom_{G (F)}  
 \left( R^{\bullet}\Gamma_{\mathrm{c}} (\Sht_{b,b'}^{(1,0)}),\pi 
 \right) = 
 -R^{\bullet}\Hom_{T (F)} 
 \left(  
 R^{\bullet}\Gamma_{\mathrm{c}} (\Sht_{T,b,b'}^{(1,0)}), \pi_{N^{\mathrm{op}}} 
 \right) \left( \frac{1}{2} \right) . 
\]
\end{lem}
\begin{proof}
By \cite[A.11 Proposition, A.12 Theorem]{CasNewnonuni}, 
\cite[Theorem 4.25]{GINsemi} 
(\cf \cite{HanHarr}) and 
\cite[III.2.7 Th\'eor\`eme, VI.9.6 Proposition]{RenReprp}, 
we have 
\begin{align*}
 R^{\bullet}\Hom_{G (F)}  
 \left( R^{\bullet}\Gamma_{\mathrm{c}} (\Sht_{b,b'}^{(1,0)}),\pi 
 \right) &= 
 R^{\bullet}\Hom_{G (F)} 
 \left( \pi^*, R^{\bullet}\Gamma_{\mathrm{c}} (\Sht_{b,b'}^{(1,0)})^* 
 \right) \\
 &= R^{\bullet}\Hom_{T (F)} 
 \left( (\pi^*)_N , 
 -R^{\bullet}\Gamma_{\mathrm{c}} (\Sht_{T,b,b'}^{(1,0)})^*  
 \right) \left( \frac{1}{2} \right)\\ 
 &= -R^{\bullet}\Hom_{T (F)} 
 \left( (\pi^*)_N , 
 R^{\bullet}\Gamma_{\mathrm{c}} (\Sht_{T,b,b'}^{(1,0)})^* 
 \right) \left( \frac{1}{2} \right) \\ 
 &= -R^{\bullet}\Hom_{T (F)} 
 \left(  
 R^{\bullet}\Gamma_{\mathrm{c}} (\Sht_{T,b,b'}^{(1,0)}), \pi_{N^{\mathrm{op}}} 
 \right) \left( \frac{1}{2} \right) . 
\end{align*}
\end{proof}

\begin{prop}\label{prop:indf}
Let $\chi_1, \chi_2 \colon F^{\times} \to \ol{\bQ}_{\ell}^{\times}$ 
be characters. 
Let $\varphi_{\chi_i} \colon W_F \to \ol{\bQ}_{\ell}^{\times}$ 
be the character corresponding to $\chi_i$. 
We put $\rho=\chi_1 \boxtimes \chi_2$ as representations of $T (F)$. 
Let $m \geq 0$ and $m/2 \geq l \geq 0$. 
We put 
\[
 b=\begin{pmatrix}
 \varpi^{l} & 0 \\ 
 0 & \varpi^{m-l} 
 \end{pmatrix}, \quad 
 b_1=\begin{pmatrix}
 \varpi^{l-1} & 0 \\ 
 0 & \varpi^{m-l} 
 \end{pmatrix}, \quad 
 b_2=\begin{pmatrix}
 \varpi^{l-1} & 0 \\ 
 0 & \varpi^{m-1-l} 
 \end{pmatrix} . 
\] 
\begin{enumerate}
\item\label{enu:neq2m}
Assume $m \neq 2l$. 
We put 
\[
 b_1'=\begin{pmatrix}
 \varpi^{l} & 0 \\ 
 0 & \varpi^{m-l-1} 
 \end{pmatrix}. 
\]
If $l=0$, then we have 
\[
 H_{\mathrm{c}}^{\bullet} (\Sht_{b,1}^{(m,0)})[\rho] = 
 (-1)^m ( \Ind_{B(F)}^{G(F)} \rho ) 
 \boxtimes 
 \varphi_{\chi_2}^m  
 \left( \frac{m}{2} \right) . 
\]
If $l \geq 1$, then we have 
\begin{align*}
 H_{\mathrm{c}}^{\bullet} & (\Sht_{b,1}^{(m,0)})[\rho ] \\ 
 &= 
 - H_{\mathrm{c}}^{\bullet} 
 (\Sht_{b_1,1}^{(m-1,0)}) [\rho ] 
 \otimes \varphi_{\chi_1} \left( - \frac{1}{2} \right) 
 -
 H_{\mathrm{c}}^{\bullet} (\Sht_{b_2,1}^{(m-2,0)})[\rho] 
 \otimes \varphi_{\chi_1} \otimes \varphi_{\chi_2} \\ 
 & \quad 
 \begin{cases}
 - H_{\mathrm{c}}^{\bullet} (\Sht_{b_1',1}^{(m-1,0)}) 
 [ \Ind_{B(F)}^{G(F)} \rho ] 
 \otimes 
 \varphi_{\chi_2} \left( \frac{1}{2} \right) 
 & \textrm{if}\  m=2l+1 \\ 
 - H_{\mathrm{c}}^{\bullet} (\Sht_{b_1',1}^{(m-1,0)}) [\rho ] 
 \otimes 
 \varphi_{\chi_2} \left( \frac{1}{2} \right) 
 & \textrm{if}\  m \geq 2l+2. 
 \end{cases}
\end{align*} 
\item\label{enu:eq2m} 
Assume $m = 2l$. 
If $l=0$, then we have 
\[
 H_{\mathrm{c}}^{\bullet} (\Sht_{b,1}^{(0,0)})[\Ind_{B(F)}^{G(F)} \rho] 
 = (\Ind_{B(F)}^{G(F)} \rho ) \boxtimes 1 . 
\]
If $l \geq 1$, then we have 
\begin{align*}
 H_{\mathrm{c}}^{\bullet} (\Sht_{b,1}^{(m,0)})[\Ind_{B(F)}^{G(F)} \rho] 
 = {} & 
 -H_{\mathrm{c}}^{\bullet} (\Sht_{b_1,1}^{(m-1,0)}) [\rho ] 
 \otimes \varphi_{\chi_1} 
 \left( -\frac{1}{2} \right) \\ 
 & -H_{\mathrm{c}}^{\bullet} (\Sht_{b_1,1}^{(m-1,0)}) [
 \rho^w \otimes \delta_B^{-1} ]
 \otimes 
 \varphi_{\chi_2} 
 \left( \frac{1}{2} \right) \\ 
 & -H_{\mathrm{c}}^{\bullet} 
 (\Sht_{b_2,1}^{(m-2,0)}) [\Ind_{B(F)}^{G(F)} \rho] 
 \otimes \varphi_{\chi_1} \otimes \varphi_{\chi_2}. 
\end{align*} 
\end{enumerate}
\end{prop}
\begin{proof}
First we show the claim \ref{enu:neq2m}. 
If $l=0$, we have 
\begin{align*}
 R^{\bullet}  \Hom_{G_b (F)} & 
 \left( R\Gamma_{\mathrm{c}} (\Sht_{b,1}^{(m,0)}),\rho \right) 
 = R^{\bullet}  \Hom_{G_b (F)} 
 \left( R\Gamma_{\mathrm{c}} (\Sht_{1,b}^{(0,-m)}),\rho \right) \\ 
 &= 
 (-1)^m R^{\bullet}  \Hom_{G_b (F)} 
 \left( \Ind_{B(F)}^{G(F)} 
 R^{\bullet} \Gamma_{\mathrm{c}} (\Sht_{T,1,b}^{(0,-m)})
 \left( \frac{m}{2} \right) \otimes \delta_b^{-1},\rho \right) \\ 
 &=
 (-1)^m \Ind_{B(F)}^{G(F)} \left( R^{\bullet} \Hom_{G_b (F)} 
 \left(  
 R^{\bullet} \Gamma_{\mathrm{c}} (\Sht_{T,b,1}^{(0,m)})
 \left( \frac{m}{2} \right) \otimes \delta_b^{-1},\rho \right) 
 \otimes \delta_B^{-1} \right) \\ 
 &=
 (-1)^m \Ind_{B(F)}^{G(F)} \left( R^{\bullet} \Hom_{G_b (F)} 
 \left(  
 R^{\bullet} \Gamma_{\mathrm{c}} (\Sht_{T,b,1}^{(0,m)})
 \left( \frac{m}{2} \right),\rho \otimes \delta_b \right) 
 \otimes \delta_B^{-1} \right) \\ 
 &= 
 (-1)^m ( \Ind_{B(F)}^{G(F)} \rho ) 
 \boxtimes 
 \varphi_{\chi_2}^m \left( \frac{m}{2} \right), 
\end{align*}
where we use $\Sht_{1,b}^{(m-1,1)}=\emptyset$ and 
\cite[Theorem 4.25]{GINsemi} at the second equality. 
We assume that $l \geq 1$. 
By Proposition \ref{prop:decomp} and Lemma \ref{lem:RHom}, 
the sum 
\[
 R^{\bullet} \Hom_{G_b (F)} 
 \left( R\Gamma_{\mathrm{c}} (\Sht_{b,1}^{(m,0)}), \rho 
 \right) + 
 R^{\bullet} \Hom_{G_b (F)} 
 \left( R\Gamma_{\mathrm{c}} (\Sht_{b,1}^{(m-1,1)}), \rho 
 \right)
\]
is equal to the sum 
\begin{align*}
 R^{\bullet} & \Hom_{G_{b_1}(F)} \left( 
 R\Gamma_{\mathrm{c}} (\Sht_{b_1,1}^{(m-1,0)}) ,
 R\Hom_{G_b (F)} 
 \left( R\Gamma_{\mathrm{c}} (\Sht_{b,b_1}^{(1,0)}),\rho \right) \otimes \delta_{b_1}^{-1} \right) \\ 
 &+R^{\bullet} \Hom_{G_{b_1'}(F)} \left( 
 R\Gamma_{\mathrm{c}} (\Sht_{b_1',1}^{(m-1,0)}) ,
 R\Hom_{G_b (F)} 
 \left( R\Gamma_{\mathrm{c}} (\Sht_{b,b_1'}^{(1,0)}),\rho \right) \otimes \delta_{b_1'}^{-1}  \right). 
\end{align*}

Since the fiber of the natural morphism 
$\Sht_{b,b_1}^{(1,0)} \to \Sht_{T,b,b_1}^{(1,0)}$ 
is isomorphic to $\mathbb{B}^{\varphi=\varpi^{m+1-2l}}$, 
we have 
%\[
% R^{\bullet} \Gamma_{\mathrm{c}} (\Sht_{b_1,b}^{(1,0)}) = 
% -R^{\bullet} \Gamma_{\mathrm{c}} (\Sht_{T,b_1,b}^{(1,0)}) 
% \left( 2l-m -\frac{1}{2} \right) . 
%\]
%Hence we have 
\[
 R^{\bullet} \Hom_{G_b (F)} 
 \left( R\Gamma_{\mathrm{c}} (\Sht_{b,b_1}^{(1,0)}),\rho \right) 
 = -(\rho \otimes \delta_B ) 
 \boxtimes \varphi_{\chi_1} \left( - \frac{1}{2} \right) .
\]
Further, we have 
\begin{align*}
 R^{\bullet} \Hom_{G_{b_1}(F)} & \left( 
 R\Gamma_{\mathrm{c}} (\Sht_{b_1,1}^{(m-1,0)}) ,
 R\Hom_{G_b (F)} 
 \left( R\Gamma_{\mathrm{c}} (\Sht_{b,b_1}^{(1,0)}),\rho \right) \otimes \delta_{b_1}^{-1} \right) \\ 
 &= - 
 R^{\bullet} \Hom_{G_{b_1}(F)} \left( 
 R\Gamma_{\mathrm{c}} (\Sht_{b_1,1}^{(m-1,0)}) ,
 \rho \right) \boxtimes \varphi_{\chi_1} \left( -\frac{1}{2} \right) . 
\end{align*}

If $m = 2l+1$, 
we have 
\[
 R^{\bullet}  \Hom_{G_b (F)} 
 \left( R\Gamma_{\mathrm{c}} (\Sht_{b,b_1'}^{(1,0)}), 
 \rho \right) 
 = 
 -\left( \Ind_{B(F)}^{G(F)} \rho \right) 
 \boxtimes 
 \varphi_{\chi_2} \left( \frac{1}{2} \right) 
\]
by the claim in the case where $l=0$. 

If $m \geq 2l+2$, 
since the fiber of the natural morphism 
$\Sht_{b,b_1'}^{(1,0)} \to \Sht_{T,b,b_1'}^{(0,1)}$ 
is isomorphic to $\mathbb{B}^{\varphi=\varpi^{m-2l}}$, 
we have 
\begin{align*}
 R^{\bullet}  \Hom_{G_b (F)} 
 \left( R\Gamma_{\mathrm{c}} (\Sht_{b,b_1'}^{(1,0)}),\rho \right)
% &= - 
% R^{\bullet}  \Hom_{G_b (F)} 
% \left( 
% R^{\bullet} \Gamma_{\mathrm{c}} (\Sht_{T,b_1',b}^{(0,1)})
% \left( \frac{1}{2}+2l -m \right),\rho \right) \\ 
 &= - 
 (\rho \otimes \delta_B ) \boxtimes 
 \varphi_{\chi_2} \left( \frac{1}{2} \right).
\end{align*}
Therefore
\begin{align*}
 R^{\bullet} & \Hom_{G_b (F)} 
 \left( R\Gamma_{\mathrm{c}} (\Sht_{b,1}^{(m,0)}), \rho 
 \right) \\ 
 = {} & 
 R^{\bullet} \Hom_{G_{b_1}(F)} \left( 
 R\Gamma_{\mathrm{c}} (\Sht_{b_1,1}^{(m-1,0)}) ,
 R\Hom_{G_b (F)} 
 \left( R\Gamma_{\mathrm{c}} (\Sht_{b,b_1}^{(1,0)}),\rho \right) \otimes \delta_{b_1}^{-1} \right) \\ 
 &+R^{\bullet} \Hom_{G_{b_1'}(F)} \left( 
 R\Gamma_{\mathrm{c}} (\Sht_{b_1',1}^{(m-1,0)}) ,
 R\Hom_{G_b (F)} 
 \left( R\Gamma_{\mathrm{c}} (\Sht_{b,b_1'}^{(1,0)}),\rho \right) \otimes \delta_{b_1'}^{-1} \right) \\ 
 &-  R^{\bullet} \Hom_{G_b (F)} 
 \left( R\Gamma_{\mathrm{c}} (\Sht_{b,1}^{(m-1,1)}), \rho 
 \right) \\ 
 = {} & - H_{\mathrm{c}}^{\bullet} (\Sht_{b_1,1}^{(m-1,0)}) 
 [\rho ] 
 \otimes \varphi_{\chi_1} \left( - \frac{1}{2} \right) 
 -
 H_{\mathrm{c}}^{\bullet} (\Sht_{b,1}^{(m-2,0)})[\rho] 
 \otimes \varphi_{\chi_1} \otimes \varphi_{\chi_2} \\ 
 &  
 \begin{cases}
 - H_{\mathrm{c}}^{\bullet} (\Sht_{b_1',1}^{(m-1,0)}) 
 [ \Ind_{B(F)}^{G(F)} \rho ] 
 \otimes 
 \varphi_{\chi_2} \left( \frac{1}{2} \right) 
 & \textrm{if}\  m=2l+1, \\ 
 - H_{\mathrm{c}}^{\bullet} (\Sht_{b_1',1}^{(m-1,0)}) [\rho ] 
 \otimes 
 \varphi_{\chi_2} \left( \frac{1}{2} \right) 
 & \textrm{if}\  m \geq 2l+2 . 
 \end{cases}
\end{align*} 

Next we show the claim \ref{enu:eq2m}. 
The claim is trivial if $l=0$. 
Assume that $l > 0$. 
We put $\pi =\Ind_{B(F)}^{G(F)} \rho$ and 
\[
 b_1'=\begin{pmatrix}
 0 & \varpi^{l-1} \\ 
 \varpi^{l} & 0 
 \end{pmatrix}. 
\]
By Proposition \ref{prop:decomp} and Lemma \ref{lem:RHom}, 
the sum 
\[
 R^{\bullet} \Hom_{G(F)} 
 \left( R\Gamma_{\mathrm{c}} (\Sht_{b,1}^{(m,0)}), \pi 
 \right) + 
 R^{\bullet} \Hom_{G(F)} 
 \left( R\Gamma_{\mathrm{c}} (\Sht_{b,1}^{(m-1,1)}), \pi 
 \right)
\]
is equal to the sum 
\begin{align*}
 R^{\bullet} & \Hom_{G_{b_1}(F)} \left( 
 R\Gamma_{\mathrm{c}} (\Sht_{b_1,1}^{(m-1,0)}) ,
 R\Hom_{G(F)} 
 \left( R\Gamma_{\mathrm{c}} (\Sht_{b,b_1}^{(1,0)}),\pi \right) \otimes \delta_{b_1}^{-1} \right) \\ 
 &+R^{\bullet} \Hom_{G_{b_1'}(F)} \left( 
 R\Gamma_{\mathrm{c}} (\Sht_{b_1',1}^{(m-1,0)}) ,
 R\Hom_{G(F)} 
 \left( R\Gamma_{\mathrm{c}} (\Sht_{b,b_1'}^{(1,0)}),\pi \right) \otimes \delta_{b_1'}^{-1} \right). 
\end{align*}
We have 
\[
 R^{\bullet} \Hom_{G(F)} 
 \left( R\Gamma_{\mathrm{c}} (\Sht_{b,b_1'}^{(1,0)}),\pi \right) =0 
\]
by \cite[Th\'eor\`eme A]{DatLTel}. 

By Lemma \ref{lem:(1,0)pi}
and the geometric lemma (\cf \cite[VI.5.1 Th\'eor\`eme]{RenReprp}), 
we have 
\begin{align*}
 & R^{\bullet} \Hom_{G_{b_1}(F)} \left( 
 R\Gamma_{\mathrm{c}} (\Sht_{b_1,1}^{(m-1,0)}) ,
 R\Hom_{G(F)} 
 \left( R\Gamma_{\mathrm{c}} (\Sht_{b,b_1}^{(1,0)}),\pi \right) \otimes \delta_{b_1}^{-1} \right) \\ 
 & = 
 -R^{\bullet} \Hom_{G_{b_1}(F)} \left( 
 R\Gamma_{\mathrm{c}} (\Sht_{b_1,1}^{(m-1,0)}) ,
 R^{\bullet}\Hom_{T (F)} 
 \left(  
 R^{\bullet}\Gamma_{\mathrm{c}} (\Sht_{T,b,b_1}^{(1,0)}), \pi_{N^{\mathrm{op}}}
 \right) \left(  \frac{1}{2} \right) \otimes \delta_{b_1}^{-1} 
 \right) \\ 
 & = 
 -R^{\bullet} \Hom_{G_{b_1}(F)} \left( 
 R\Gamma_{\mathrm{c}} (\Sht_{b_1,1}^{(m-1,0)}) ,
 R^{\bullet}\Hom_{T (F)} 
 \left( 
 R^{\bullet}\Gamma_{\mathrm{c}} (\Sht_{T,b,b_1}^{(1,0)}), 
 (\rho \otimes \delta_B )+ \rho^w 
 \right) \left( \frac{1}{2} \right) \otimes \delta_{b_1}^{-1} 
 \right) \\ 
 & = 
 -R^{\bullet} \Hom_{G_{b_1}(F)} \left( 
 R\Gamma_{\mathrm{c}} (\Sht_{b_1,1}^{(m-1,0)}) ,
 \rho \right) 
 \otimes \varphi_{\chi_1} 
 \left( -\frac{1}{2} \right) \\ 
 & \quad \ -R^{\bullet} \Hom_{G_{b_1}(F)} \left( 
 R\Gamma_{\mathrm{c}} (\Sht_{b_1,1}^{(m-1,0)}) ,
 \rho^w \otimes \delta_B^{-1} 
 \right) \otimes 
 \varphi_{\chi_2} 
 \left( \frac{1}{2} \right) . 
\end{align*}
Hence 
\begin{align*}
 H_{\mathrm{c}}^{\bullet} (\Sht_{b,1}^{(m,0)}) [ \pi ] 
 & =  
 -H_{\mathrm{c}}^{\bullet} (\Sht_{b,1}^{(m-1,1)}) [\pi] 
 -H_{\mathrm{c}}^{\bullet} (\Sht_{b_1,1}^{(m-1,0)}) [\rho ] 
 \otimes \varphi_{\chi_1} 
 \left( -\frac{1}{2} \right) \\ 
 & \quad \ -H_{\mathrm{c}}^{\bullet} (\Sht_{b_1,1}^{(m-1,0)}) [
 \rho^w \otimes \delta_B^{-1} ]
 \otimes 
 \varphi_{\chi_2} 
 \left( \frac{1}{2} \right) . 
\end{align*}
Therefore we obtain the claim. 
\end{proof}

By Proposition \ref{prop:indf}, 
we can calculate 
$H_{\mathrm{c}}^{\bullet} (\Sht_{b,1}^{(m,0)})[\rho]$ 
and 
$H_{\mathrm{c}}^{\bullet} (\Sht_{b,1}^{(m,0)})[\Ind_{B(F)}^{G(F)} \rho]$ 
in Proposition \ref{prop:indf} inductively. 
We do not pursue the explicit formula here, 
but record the following corollary. 

\begin{cor}\label{cor:parcomb}
The $\GL_2 (F)$-representations 
$H_{\mathrm{c}}^{\bullet} (\Sht_{b,1}^{(m,0)})[\rho]$ 
and 
$H_{\mathrm{c}}^{\bullet} (\Sht_{b,1}^{(m,0)})[\Ind_{B(F)}^{G(F)} \rho]$ 
in Proposition \ref{prop:indf} 
are linear combinations of proper parabolic inductions. 
\end{cor}
\begin{proof}
This follows from 
Proposition \ref{prop:indf} 
by induction. 
\end{proof}

\begin{prop}\label{prop:basind}
We put 
\[
  b_1=
 \begin{pmatrix}
 0 & 1 \\ 
 \varpi & 0 
 \end{pmatrix} 
\]
and 
$b_m=b_1^m$ for $m \in \bZ$. 
For an odd integer $m$, 
we put 
\[
  b_m' =
 \begin{pmatrix}
 \varpi^{\frac{m-1}{2}} & 0 \\ 
 0 & \varpi^{\frac{m+1}{2}} 
 \end{pmatrix}. 
\]
Assume that $m \geq 2$. 
If $m$ is odd or $\varphi$ is cuspidal, 
we have 
\begin{align*}
 H_{\mathrm{c}}^{\bullet} (\Sht_{b_{m},1}^{(m,0)})[\pi_{b_{m}}] 
 =
 H_{\mathrm{c}}^{\bullet} (\Sht_{b_{m-1},1}^{(m-1,0)})[\pi_{b_{m-1}}] \otimes \varphi 
 - 
 H_{\mathrm{c}}^{\bullet} (\Sht_{b_{m-2},1}^{(m-2,0)})[\pi_{b_{m-2}}] \otimes 
 (r_{(1,1)} \circ \varphi) . 
\end{align*} 
If $m$ is even and $\varphi$ is not cuspidal, 
we have 
\begin{align*}
 H_{\mathrm{c}}^{\bullet} (\Sht_{b_{m},1}^{(m,0)})[\pi_{b_{m}}] 
 = {} & 
 H_{\mathrm{c}}^{\bullet} (\Sht_{b_{m-1},1}^{(m-1,0)})[\pi_{b_{m-1}}] \otimes \varphi 
 - 
 H_{\mathrm{c}}^{\bullet} (\Sht_{b_{m-2},1}^{(m-2,0)})[\pi_{b_{m-2}}] \otimes 
 (r_{(1,1)} \circ \varphi) \\ 
 & - 
 H_{\mathrm{c}}^{\bullet} (\Sht_{b_{m-1}',1}^{(m-1,0)})
 [\chi \boxtimes \chi] 
 \otimes \varphi_{\chi} \left( -\frac{1}{2} \right) 
\end{align*} 
where $\chi$ is a character of $F^{\times}$ such that 
$\pi_{b_{m}} \simeq \mathrm{St}_{\chi}$. 
\end{prop}
\begin{proof}
Assume that $m$ is odd. 
By Proposition \ref{prop:decomp} and Lemma \ref{lem:RHom}, 
the sum 
\[
 R^{\bullet} \Hom_{\GL_2 (F)} 
 \left( R\Gamma_{\mathrm{c}} (\Sht_{b_{m},1}^{(m,0)}), \pi_{b_{m}} 
 \right) + 
 R^{\bullet} \Hom_{\GL_2 (F)} 
 \left( R\Gamma_{\mathrm{c}} (\Sht_{b_{m},1}^{(m-1,1)}), \pi_{b_{m}} 
 \right)
\]
is equal to 
\begin{align*}
 R^{\bullet} & \Hom_{G_{b_{m-1}}(F)} \left( 
 R\Gamma_{\mathrm{c}} (\Sht_{b_{m-1},1}^{(m-1,0)}) ,
 R\Hom_{\GL_2 (F)} 
 \left( R\Gamma_{\mathrm{c}} (\Sht_{b_{m},b_{m-1}}^{(1,0)}),\pi_{b_{m}} \right) \right) . 
\end{align*}
Hence the claim follows from Corollary \ref{cor:minu}. 

Assume that $m$ is even. 
By Proposition \ref{prop:decomp} and Lemma \ref{lem:RHom}, 
the sum 
\[
 R^{\bullet} \Hom_{\GL_2 (F)} 
 \left( R\Gamma_{\mathrm{c}} (\Sht_{b_{m},1}^{(m,0)}), \pi_{b_{m}} 
 \right) + 
 R^{\bullet} \Hom_{\GL_2 (F)} 
 \left( R\Gamma_{\mathrm{c}} (\Sht_{b_{m},1}^{(m-1,1)}), \pi_{b_{m}} 
 \right)
\]
is equal to the sum 
\begin{align*}
 R^{\bullet} & \Hom_{G_{b_{m-1}}(F)} \left( 
 R\Gamma_{\mathrm{c}} (\Sht_{b_{m-1},1}^{(m-1,0)}) ,
 R\Hom_{\GL_2 (F)} 
 \left( R\Gamma_{\mathrm{c}} (\Sht_{b_{m},b_{m-1}}^{(1,0)}),\pi_{b_{m}} \right) \right) \\ 
 &+R^{\bullet} \Hom_{G_{b_{m-1}'}(F)} \left( 
 R\Gamma_{\mathrm{c}} (\Sht_{b_{m-1}',1}^{(m-1,0)}) ,
 R\Hom_{\GL_2 (F)} 
 \left( R\Gamma_{\mathrm{c}} (\Sht_{b_m,b_{m-1}'}^{(1,0)}),\pi_{b_{m}} \right) 
 \otimes \delta_{B}^{-1} \right). 
\end{align*}
Hence, by Corollary \ref{cor:minu}, 
it suffices to show that 
%\begin{align*}
\[
 R\Hom_{\GL_2 (F)} 
 \left( R\Gamma_{\mathrm{c}} (\Sht_{b_m,b_{m-1}'}^{(1,0)}),\pi_{b_{m}} \right) = 
\begin{cases}
 0 & \textrm{if $\varphi$ is cuspidal,} \\ 
 -((\chi \boxtimes \chi) \otimes \delta_B )
 \otimes \varphi_{\chi} \left( -\frac{1}{2} \right)  & 
 \textrm{if $\varphi$ is not cuspidal.}
\end{cases}
\]
%\end{align*}
By Lemma \ref{lem:(1,0)pi}, 
we have 
\begin{align*}
 &R^{\bullet}\Hom_{\GL_2 (F)} 
 \left( R\Gamma_{\mathrm{c}} (\Sht_{b_{m},b_{m-1}'}^{(1,0)}),\pi_{b_{m}} 
 \right) \\ 
&= -R^{\bullet}\Hom_{T (F)} 
 \left(  
 R^{\bullet}\Gamma_{\mathrm{c}} (\Sht_{T,b_{m},b_{m-1}'}^{(1,0)}), (\pi_{b_{m}})_{N^{\mathrm{op}}} 
 \right) \left( \frac{1}{2} \right) . 
\end{align*}
Hence the claim follows from 
$(\mathrm{St}_{\chi})_{N^{\mathrm{op}}} \simeq (\chi \boxtimes \chi) \otimes \delta_B$. 
\end{proof}

\begin{prop}\label{prop:basdiv}
We put 
\[
  b_1=
 \begin{pmatrix}
 0 & 1 \\ 
 \varpi & 0 
 \end{pmatrix} 
\]
and 
$b_m=b_1^m$ for $m \in \bZ$. 
For $m \geq 1$, we have 
\begin{align*}
 R^{\bullet} & \Hom_{G_{b_{m}} (F)} 
 \left( R\Gamma_{\mathrm{c}} (\Sht_{b_{m},b_{-1},}^{(m+1,0)}), \pi_{b_{m}} 
 \right) \\ 
 &= 
 R^{\bullet} \Hom_{\GL_2(F)} \left( 
 R\Gamma_{\mathrm{c}} (\Sht_{1,b_{-1}}^{(1,0)}) , 
 R\Hom_{G_{b_{m}} (F)} 
 \left( R\Gamma_{\mathrm{c}} (\Sht_{b_{m},1}^{(m,0)}),\pi_{b_{m}} \right) \right) \\ 
 & \quad -  R^{\bullet} \Hom_{G_{b_{m-2}} (F)} 
 \left( R\Gamma_{\mathrm{c}} (\Sht_{b_{m-2},b_{-1}}^{(m-1,0)}), \pi_{b_{m-2}} 
 \right) \otimes (r_{(1,1)} \circ \varphi) . 
\end{align*}
\end{prop}
\begin{proof}
This follows from 
Proposition \ref{prop:decomp} and Lemma \ref{lem:RHom}. 
\end{proof}

\begin{thm}\label{thm:nonmin}
Assume that $n=2$. Then 
Conjecture \ref{conj:Kot} is true if 
$\kappa(b)$ is odd or $\varphi$ is cuspidal. 
\end{thm}
\begin{proof}
We put 
\[
  b_1=
 \begin{pmatrix}
 0 & 1 \\ 
 \varpi & 0 
 \end{pmatrix}. 
\]
To show the claim, 
we may assume that 
$\mu=(m,0)$ for some $m \geq 0$ and 
$b$ is $1$ or $b_1$ 
by twisting. 

Assume that $\varphi$ is cuspidal. 
If $b=1$, we can show the claim by induction 
using Proposition \ref{prop:basind}. 
If $b=b_1$, 
we can show the claim by induction 
using Proposition \ref{prop:basdiv} and 
the case for $b=1$. 

It remains to treat the case where 
$\varphi$ is not cuspidal and $b=b_1$. 
First, we can show that 
\[
 H_{\mathrm{c}}^{\bullet} (\Sht_{b',1}^{(m,0)})[\pi_{b'}] - 
 \pi_{1} \boxtimes (r_{(m,0)} \circ \varphi) 
\]
is a linear combination of proper parabolic inductions 
as representations of $\GL_2 (F)$ 
using Corollary \ref{cor:parcomb} 
and Proposition \ref{prop:basind}. 
Hence, the claim follows from 
Proposition \ref{prop:basdiv} and 
\cite[Th\'eor\`eme A]{DatLTel}. 
\end{proof}

On the other hand, the following example shows that 
Conjecture \ref{conj:Kot} is not true if 
$\mu$ is not minuscule and $\varphi$ is not cuspidal. 

\begin{eg}\label{eg:(2,0)}
Let $\mu=(2,0)$ and 
$b$ be a basic element such that $\kappa(b)=2$. 
Assume that $\varphi$ is not cuspidal and take 
a character $\chi$ of $F^{\times}$ 
such that $\pi_1 \simeq \mathrm{St}_{\chi}$. 
Then we have 
\begin{align*}
 R^{\bullet} & \Hom_{G_b (F)} 
 \left( R\Gamma_{\mathrm{c}} (\Sht_{b,1}^{\mu}), \pi_b 
 \right) \\ 
 &=
 \pi_1 \boxtimes (r_{\mu} \circ \varphi) 
 + 
 \left( \Ind_{B(F)}^{\GL_2 (F)} 
  (\chi \boxtimes \chi ) \right) 
 \boxtimes (r_{(1,1)} \circ \varphi)
\end{align*} 
by Proposition \ref{prop:indf} 
and Proposition \ref{prop:basind}. 
\end{eg}

\begin{rem}
Example \ref{eg:(2,0)} 
is compatible with the main theorem of 
\cite{HKWKotloc}, 
since the representation 
$\Ind_{B(F)}^{\GL_2 (F)} 
 (\chi \boxtimes \chi)$ 
has trace $0$ on regular elliptic elements. 
\end{rem}

\begin{rem}\label{rem:errexp}
The error term in Example \ref{eg:(2,0)} 
supports that 
the expectation in \cite[Remark 4.6]{FarGover} 
is true. 
\end{rem}

\begin{eg}\label{eg:(3,0)}
Let $\chi_1, \chi_2 \colon F^{\times} \to \ol{\bQ}_{\ell}^{\times}$ 
be characters. 
Let $\varphi_{\chi_i} \colon W_F \to \ol{\bQ}_{\ell}^{\times}$ 
be the character corresponding to $\chi_i$. 
We put 
$b=\begin{pmatrix}
 \varpi & 0 \\ 
 0 & \varpi^{2} 
 \end{pmatrix}$ 
and $\mu=(3,0)$. 
We put $\rho=\chi_1 \boxtimes \chi_2$ as representations of $T (F)$. 
Then we have 
\begin{align*}
 H^{\bullet} (\Sht_{G,b,1}^{\mu})[\rho] 
 = 
 -(\Ind_{B(F)}^{G(F)} \rho )
 \boxtimes \varphi_{\chi_1} 
 \otimes 
 \varphi_{\chi_2}^2 \left( \frac{1}{2} \right) . 
\end{align*} 
\end{eg}
\begin{proof}
We put 
\[
 b_1=
 \begin{pmatrix}
 1 & 0 \\ 
 0 & \varpi^{2} 
 \end{pmatrix}, 
 \quad  
 b_1'=
 \begin{pmatrix}
 \varpi & 0 \\ 
 0 & \varpi 
 \end{pmatrix}, 
 \quad 
 b_2=
 \begin{pmatrix}
 1 & 0 \\ 
 0 & \varpi 
 \end{pmatrix}. 
\]
By Proposition \ref{prop:indf}, 
we have 
\begin{align*}
 H_{\mathrm{c}}^{\bullet} & (\Sht_{b,1}^{(3,0)})[\rho] \\ 
 =& 
 - H_{\mathrm{c}}^{\bullet} (\Sht_{b_1,1}^{(2,0)}) 
 [\rho ] 
 \otimes \varphi_{\chi_1} \left( -\frac{1}{2} \right) 
 -
 H_{\mathrm{c}}^{\bullet} (\Sht_{b_2,1}^{(1,0)})[\rho] 
 \otimes \varphi_{\chi_1} \otimes \varphi_{\chi_2} \\ 
 & 
 - H_{\mathrm{c}}^{\bullet} (\Sht_{b_1',1}^{(2,0)}) 
 [\Ind_{B(F)}^{G(F)} \rho ] 
 \otimes 
 \varphi_{\chi_2} \left( \frac{1}{2} \right)\\ 
 =& 
 -\Ind_{B(F)}^{G(F)} (\rho) 
 \boxtimes \varphi_{\chi_1} 
 \otimes 
 \varphi_{\chi_2}^2 \left( \frac{1}{2} \right) 
 +(\Ind_{B(F)}^{G(F)} \rho )
 \boxtimes \varphi_{\chi_1} 
 \otimes 
 \varphi_{\chi_2}^2 \left( \frac{1}{2} \right) \\ 
 &+H_{\mathrm{c}}^{\bullet} (\Sht_{b_2,1}^{(1,0)}) [\rho ] 
 \otimes \varphi_{\chi_1} 
 \otimes 
 \varphi_{\chi_2} +H_{\mathrm{c}}^{\bullet} (\Sht_{b_2,1}^{(1,0)}) 
 [\rho^w \otimes \delta_B^{-1}]
 \otimes 
 \varphi_{\chi_2}^2 
 \left( 1 \right) \\ 
 & +H_{\mathrm{c}}^{\bullet} 
 (\Sht_{1,1}^{(0,0)})[\Ind_{B(F)}^{G(F)} \rho ] 
 \otimes \varphi_{\chi_1} \otimes \varphi_{\chi_2}^2 
 \left( \frac{1}{2} \right), \\ 
 = & 
 -\Ind_{B(F)}^{G(F)} (\rho^w \otimes \delta_B^{-1}) 
 \boxtimes \varphi_{\chi_1} \otimes 
 \varphi_{\chi_2}^2 
 \left( \frac{1}{2} \right) . 
\end{align*}
Therefore we obtain the claim. 
\end{proof}

\begin{rem}\label{rem:HV30}
We use notation in Example \ref{eg:(3,0)}. 
We define $I_{b,\mu,T}$ in the same way as \cite[(31)]{RVlocSh}. 
Then we have $I_{b,\mu,T}=\emptyset$. 
Therefore 
Example \ref{eg:(3,0)} shows that 
the non-minuscule generalization of 
\cite[Conjecture 8.5]{RVlocSh} does not hold as it is. 
We note that $([b],\mu)$ is not Hodge--Newton reducible 
(\cf \cite[Definition 4.28]{RVlocSh}). 
\end{rem}

%\bibliographystyle{test2}
%\bibliography{reference}

\begin{thebibliography}{HKW22}
	\providecommand{\url}[1]{\texttt{#1}}
	\providecommand{\urlprefix}{URL }
	\providecommand{\eprint}[2][]{\url{#2}}
	
	\bibitem[ALB21]{AnLBAvGLn}
	J.~Ansch\"utz and A.-C. Le~Bras, Averaging functors in Fargues' program for
	$\mathrm{GL}_n$, 2021, arXiv:2104.04701.
	
	\bibitem[BZ76]{BeZeRepGL}
	I.~N. Bern\v{s}te\u{\i}n and A.~V. Zelevinski\u{\i}, Representations of the
	group {$GL(n,F),$} where {$F$} is a local non-{A}rchimedean field, Uspehi
	Mat. Nauk 31 (1976), no.~3, 5--70.
	
	\bibitem[Cas82]{CasNewnonuni}
	W.~Casselman, A new nonunitarity argument for {$p$}-adic representations, J.
	Fac. Sci. Univ. Tokyo Sect. IA Math. 28 (1982), no.~3, 907--928.
	
	\bibitem[CS17]{CaScGenSV}
	A.~Caraiani and P.~Scholze, On the generic part of the cohomology of compact
	unitary {S}himura varieties, Ann. of Math. (2) 186 (2017), no.~3, 649--766.
	
	\bibitem[Dat07]{DatLTel}
	J.-F. Dat, Th\'{e}orie de {L}ubin-{T}ate non-ab\'{e}lienne et
	repr\'{e}sentations elliptiques, Invent. Math. 169 (2007), no.~1, 75--152.
	
	\bibitem[Far16]{FarGover}
	L.~Fargues, Geometrization of the local {L}anglands correspondence: An
	overview, 2016, arXiv:1602.00999.
	
	\bibitem[FS21]{FaScGeomLLC}
	L.~Fargues and P.~Scholze, Geometrization of the local {L}anglands
	correspondence, 2021, arXiv:2102.13459.
	
	\bibitem[GI16]{GINsemi}
	I.~Gaisin and N.~Imai, Non-semi-stable loci in {H}ecke stacks and {F}argues'
	conjecture, 2016, arXiv:1608.07446.
	
	\bibitem[Han21a]{HanHarr}
	D.~Hansen, Moduli of local shtukas and {H}arris's conjecture, Tunis. J. Math. 3
	(2021), no.~4, 749--799.
	
	\bibitem[Han21b]{HansclocSh}
	D.~Hansen, On the supercuspidal cohomology of basic local Shimura varieties,
	2021, preprint.
	
	\bibitem[HKW22]{HKWKotloc}
	D.~Hansen, T.~Kaletha and J.~Weinstein, On the {K}ottwitz conjecture for local
	shtuka spaces, Forum Math. Pi 10 (2022), Paper No. e13, 79.
	
	\bibitem[Ima20]{ImaLLCell}
	N.~Imai, Local {L}anglands correspondences in $\ell$-adic coefficients, 2020,
	arXiv:2003.14154.
	
	\bibitem[Kot85]{KotIso}
	R.~E. Kottwitz, Isocrystals with additional structure, Compositio Math. 56
	(1985), no.~2, 201--220.
	
	\bibitem[Kot97]{KotIsoII}
	R.~E. Kottwitz, Isocrystals with additional structure. {II}, Compositio Math.
	109 (1997), no.~3, 255--339.
	
	\bibitem[Pra19]{PraMVWinv}
	D.~Prasad, Generalizing the {MVW} involution, and the contragredient, Trans.
	Amer. Math. Soc. 372 (2019), no.~1, 615--633.
	
	\bibitem[Rap95]{RapNAper}
	M.~Rapoport, Non-{A}rchimedean period domains, in Proceedings of the
	{I}nternational {C}ongress of {M}athematicians, {V}ol. 1, 2 ({Z}\"{u}rich,
	1994), Birkh\"{a}user, Basel, 1995 pp. 423--434.
	
	\bibitem[Ren10]{RenReprp}
	D.~Renard, Repr\'esentations des groupes r\'eductifs {$p$}-adiques, vol.~17 of
	Cours Sp\'ecialis\'es, Soci\'et\'e Math\'ematique de France, Paris, 2010.
	
	\bibitem[RV14]{RVlocSh}
	M.~Rapoport and E.~Viehmann, Towards a theory of local {S}himura varieties,
	M\"unster J. Math. 7 (2014), no.~1, 273--326.
	
	\bibitem[Sch17]{SchEtdia}
	P.~Scholze, Etale cohomology of diamonds, 2017, arXiv:1709.07343.
	
	\bibitem[SW20]{ScWeBLp}
	P.~{Scholze} and J.~{Weinstein}, {Berkeley lectures on \(p\)-adic geometry},
	vol. 207, Princeton, NJ: Princeton University Press, 2020.
	
\end{thebibliography}

\noindent
Naoki Imai\\
Graduate School of Mathematical Sciences, The University of Tokyo, 
3-8-1 Komaba, Meguro-ku, Tokyo, 153-8914, Japan \\
naoki@ms.u-tokyo.ac.jp 

\end{document}